\documentclass[11pt]{amsart}
\usepackage{amsmath,amssymb,a4wide,scalerel}
\usepackage{mathabx}
\usepackage{color}
\usepackage{xcolor,graphicx}
\usepackage{mathrsfs}
\usepackage{tikz,pgfplots}
\usepackage{subcaption}
\usepackage[normalem]{ulem}
\usepackage{hyperref}
\usepackage{dsfont}
\usepackage{comment}

\newtheorem{theorem}{Theorem}
\newtheorem{lemma}[theorem]{Lemma}
\newtheorem{proposition}[theorem]{Proposition}
\newtheorem{remark}[theorem]{Remark}

\newtheorem{assumption}{Assumption}

\newcommand{\Falpha}{F_\alpha}
\newcommand{\dt}{\tau}
\newcommand{\uk}{{\underline{k}}}

\newcommand{\R}{{\mathbb R}}
\newcommand{\N}{{\mathbb N}}
\newcommand{\E}{{\mathbb E}}

\newcommand{\e}{\mathrm{e}}
\newcommand{\abs}[1]{\lvert #1 \rvert}
\providecommand{\norm}[1]{\left\lVert#1\right\rVert}
\newcommand*\diff{\mathop{}\!\mathrm{d}}
\newcommand{\COV}{\operatorname{Cov}}
\newcommand{\x}{\tilde{x}}
\newcommand{\y}{\tilde{y}}

\def\msquare{\mathord{\scalerel*{\Box}{gX}}}

\DeclareMathOperator{\diag}{diag}

\usepackage{color}

\author{Charles-Edouard Br\'ehier}
              \address{Universite de Pau et des Pays de l'Adour, E2S UPPA, CNRS, LMAP, Pau, France}
              \email{charles-edouard.brehier@univ-pau.fr}

\author{David Cohen}
              \address{Department of Mathematical Sciences,
              Chalmers University of Technology and University of Gothenburg, 41296~Gothenburg, Sweden}
              \email{\tt david.cohen@chalmers.se}

\author{Llu\'{i}s Quer-Sardanyons}
             \address{Department of Mathematics, Universitat Aut\`onoma de Barcelona, 08193~Bellaterra, Spain}
              \email{\tt lluis.quer@uab.cat}

\author{Johan Ulander}
              \address{Department of Mathematical Sciences,
              Chalmers University of Technology and University of Gothenburg, 41296~Gothenburg, Sweden}
              \curraddr{Swedish Defence Research Agency (FOI), 58330~Link\"oping, Sweden}
              \email{\tt johanul@chalmers.se}
              \email{\tt johan.ulander@foi.se}

\begin{document}

\title[Stochastic exponential integrators for SPDEs driven by Riesz noise]{Analysis of an exponential integrator for stochastic PDE{\tiny s} driven by Riesz noise}

\begin{abstract}
We present and study an explicit exponential integrator for parabolic SPDEs in any dimension
driven by a Gaussian noise which is white in time and with spatial correlation given by a Riesz kernel.
Under assumptions on the coefficients of the SPDE, we prove strong error bounds and exhibit how the rate
of convergence depends on the exponent in the Riesz kernel.
Finally, numerical experiments in spatial dimensions $1$ and $2$ are provided in order to confirm
our convergence results.
\end{abstract}

\maketitle
{\small\noindent

Stochastic heat equations. Riesz kernel. Stochastic exponential integrator.
$L^p(\Omega)$-convergence.

\section{Introduction}
The analysis of convergence of numerical schemes for stochastic partial differential equations (SPDEs)
started with the seminal paper \cite{MR1341554}. In this reference, the authors showed the convergence in
probability of an implicit numerical scheme to the solution of a stochastic heat equation with a nonlinear term
and driven by additive space-time white noise defined on the interval $[0,1]$. The literature on (strong) convergence
of numerical schemes to solutions of SPDEs, in particular of the parabolic type, in dimension $1$ is
now well established, see e.g \cite{MR3308418,MR1644183,MR1873517,MR1803132,MR2471778,MR1953619,MR2600932,MR4050540,MR2211047,MR4179715,MR2646103,MR3047942,MR2147242,MR3327065,
MR3649432,MR4092279,MR3984308,MR4049401,MR4576086,MR4829617,MR4648763,MR4780408,djurdjevac2024higherorderapproximationnonlinear,MR4714476,MR4858504} and references therein.

Moving beyond the one dimensional case and focusing on the random field setting, convergence results of numerical methods for SPDEs in dimension larger than $1$ are currently lacking. Without being exhaustive, we refer the interested reader to \cite{cw00,MR2147242,MR3459554}. Indeed, the results presented in our paper are closely related to those obtained in \cite{MR2147242}. More precisely, in the latter reference the authors study the speed of convergence of explicit and implicit (semi-) discretization schemes for the following semilinear stochastic heat equation driven by Riesz noise with parameter $\alpha\in(0,2\wedge d)$, where $d\geq1$ is the space dimension:
\begin{equation}
\label{probIntro}
\begin{cases}
\frac{\partial u}{\partial t}(t,x)=\Delta u(t,x)+b(t,x,u(t,x))+\sigma(t,x,u(t,x))\dot{\Falpha}(t,x),\\
u(0,x)=u_0(x)\quad\text{for $x\in (0,1)^d$},
\end{cases}
\end{equation}
with homogeneous Dirichlet boundary conditions. Here, $\Delta$ is the Laplacian operator on $[0,1]^d$ with Dirichlet boundary conditions, $\dot{\Falpha}$ denotes a Gaussian noise which is white in time and has a spatial correlation given by a Riesz kernel (see Section \ref{sec:noise} for the precise definition), $u_0$ is a continuous function on $[0,1]^d$ and the functions $b,\sigma$ satisfy some conditions (see Section \ref{sec:ass}).

Our main objective is to analyze the rate of strong convergence of an explicit stochastic exponential integrator for the SPDE~\eqref{probIntro}.
The main contributions of the present work are twofold:

First, we prove that the strong rate of convergence of the stochastic exponential integrator is $(\frac12-\frac\alpha4)^-$ (see Theorem~\ref{thm:main} for the precise statement). This rate of convergence is optimal, in the sense that it coincides with the time regularity of the solution $u$ of equation \eqref{probIntro} (see Proposition \ref{prop:reg}). Moreover, the obtained rate of convergence coincides with those of the explicit and implicit Euler--Maruyama schemes considered in \cite[Thm. 3.4. (ii) and (iii)]{MR2147242}. On top of that,
our error estimates are uniform with respect to time and space (see \eqref{eq:main}) and require only the minimal assumption of the initial condition $u_0$ being a continuous function. We also point out that the considered stochastic exponential integrator provides two distinct benefits over the numerical schemes studied in \cite{MR2147242}: it avoids the CFL-related step-size restriction of the explicit Euler--Maruyama scheme and it has advantages in implementation over the implicit Euler--Maruyama scheme (see Section~\ref{sec-num} for numerical illustrations).

The second important contribution of the present paper is that we numerically confirm the theoretically derived strong convergence rates in dimensions $d=1$ and $d=2$ (see Section~\ref{sec-num}). To the best of our knowledge, no previous references in the literature have shown such numerical results in dimension $2$ in the random field setting.

Before elaborating a bit more on the above described contributions, let us mention that
exponential integrators have a long history in the numerical analysis of deterministic ordinary and partial differential equations,
see the review \cite{MR2652783}. It is also not the first time that this kind of time integrators have been applied in the context of SPDEs,
see for instance \cite{MR2578878,MR2728063,MR3919602,MR4568444,MR4585358,MR4756577,MR4939608,hajiali2025multiindexmontecarlomethod} and the references above. In addition, we refer the reader to \cite{MR4050540}, where the authors consider a one-dimensional stochastic heat equation of the form~\eqref{probIntro} but driven by space-time white noise. The case of stochastic wave equations has been addressed in
\cite{cls13, MR3276429, MR3353942, MR3463447, ACLW16,MR3573128,MR4462619,MR4827939,MR4931855,MR4846352}. Furthermore, in the context of stochastic Schr\"odinger equations, there have been advances in, e.g., \cite{MR2152256,CD17,MR3771721,MR3884775,MR4744301,MR4743426,MR4830442}.

The proof of our main result (Theorem \ref{thm:main}) on the strong rate of convergence for the stochastic exponential integrator is carried out using two important ingredients. On the one hand, we apply well-known estimates for stochastic integrals with respect to our Riesz noise. Here, we adapt the theory developed by Dalang in the seminal paper \cite{Dalang99} to our setting, in which we consider our SPDE in the bounded domain $[0,1]^d$. On the other hand, most of the technical effort has been focused on proving precise estimates involving the Green function $G_d$ of the stochastic heat equation in $[0,1]^d$ with Dirichlet boundary conditions (see Lemmas~\ref{lem:GdEstimates-interpoled} and \ref{lem:aux} for the precise statements), and using them sharply throughout the proof of our main theorem. The proofs of some of those results, which involve norms of $G_d$ in the Hilbert space determined by the space covariance of the noise $\Falpha$, have been sketched in \cite{MR2147242}. For the sake of completeness, we have added a detailed proof of all estimates in Appendix \ref{app:aux}. Finally, we remark that, conveniently used, this kind of bounds allows us to get our rate of convergence uniformly with respect to time and space, with only assuming that $u_0$ is continuous.

The paper is organized as follows. In Section \ref{sec-sett}, we present the main notation that will be used throughout the paper, we describe the Riesz noise $\Falpha$ together with the corresponding stochastic integral, we provide some properties and estimates involving the Green function $G_d$ and we establish the main assumptions on the coefficients $b$ and $\sigma$. Section \ref{sec-exact} is devoted to defining the mild solution to equation \eqref{probIntro} and prove some of its properties. In Section \ref{sec-conv}, the continuous-time mild formulation of the stochastic exponential integrator is presented and the main result of the paper is stated (Theorem \ref{thm:main}). Section \ref{sec-num} is devoted to present the numerical experiments in dimensions $d=1$ and $d=2$. Finally, in Section \ref{sec-proofs} we give the proof of Theorem \ref{thm:main}.  Appendix \ref{app:aux} contains detailed proofs of the main properties of the Green function $G_d$, while Appendices \ref{app:1Dnoise} and \ref{app:2Dnoise} contain details on the implementation on the covariance matrix of the noise in dimension $d=1$ and $d=2$, respectively.

\section{Setting}\label{sec-sett}

In this section, we first introduce notation in Subsection~\ref{sec:notation}, then we provide the definition of the Riesz noise and properties of the associated stochastic integral in Subsection~\ref{sec:noise}. Next, we present properties of the heat kernel in Subsection~\ref{sec:kernel}. Finally, in Subsection~\ref{sec:ass}, we state regularity assumptions on the coefficients $b$ and $\sigma$ appearing in the SPDE~\eqref{probIntro}.

\subsection{Notation}\label{sec:notation}

The set of positive integers is denoted by $\N=\{1,2,\ldots\}$. Let $Q=(0,1)^d$ and $\overline{Q}=[0,1]^d$, where $d\in\N$ is the spatial dimension. The Euclidean norm in $\R^d$ is denoted by $|\cdot|$. Let $\mathcal{B}(Q)$ denote the set of Borel subsets of $Q$.

Let $\mathcal{C}^0(\overline{Q})$ denote the space of real-valued continuous functions on $\overline{Q}$. For any $v\in\mathcal{C}^{0}(\overline{Q})$, set
\[
\|v\|_\infty=\underset{x\in\overline{Q}}\max~|v(x)|.
\]
Moreover, let
\[
\mathcal{C}_0^0(\overline{Q})=\{v\in\mathcal{C}^0(\overline{Q});\,v(x)=0,\quad\forall~x\in\partial Q\}.
\]
Let $L^2(Q)$ denote the space of square-integrable real-valued functions on $Q$. In addition, let $\mathcal{D}(Q)$ denote the space of real-valued smooth functions defined on $Q$.

In this article, we consider the Laplace operator $\Delta$ on the $d$-dimensional domain $Q$ with homogeneous Dirichlet boundary conditions on $\partial Q=\overline{Q}\setminus Q$. It is well-known that this operator can be diagonalized. For all $\uk=(k_1,\ldots,k_d)\in\N^{d}$, define $|\uk|^2=  \displaystyle\sum_{j=1}^d |k_j|^{2} $. For all $j\in\N$, set $\varphi_j(\cdot)=\sqrt{2}\sin(j\pi\cdot)$. For all $\uk\in\N^d$, define
\[
\varphi_{\uk}(x)=\prod_{j=1}^{d}\varphi_{k_j}(x_{j}),\qquad \forall~x=(x_1,\ldots,x_d)\in\overline{Q}.
\]
Then $\bigl(\varphi_{\uk}\bigr)_{\uk\in\N^d}$ is a complete orthonormal system of $L^2(Q)$. Moreover, for any $\uk\in\N^d$, the function $\varphi_{\uk}$ is an eigenfunction of the operator $\Delta$, with associated eigenvalue $-|\uk|^2\pi^2$.

Let $(\Omega,\mathcal{F},\mathbb{P})$ be a probability space. The expectation operator is denoted by $\mathbb{E}[\cdot]$.
For any $p\in[1,\infty)$ and any real-valued random variable $u$, we define
$\vvvert u\vvvert_p=\bigl(\E[|u|^p]\bigr)^{\frac1p}\in[0,\infty]$. Let $L^p(\Omega)$ denote the space of real-valued random variables such that $\vvvert u \vvvert_p<\infty$.

In the sequel, $C$ denotes a constant that may vary from line to line. Subscripts on $C$ may be employed to indicate dependence with respect to parameters.

\subsection{Riesz noise}\label{sec:noise}

In this subsection, we introduce the Riesz noise $\Falpha$, depending on the parameter $\alpha\in(0,2 \wedge d)$,
and we provide the main tools to define stochastic integrals with respect to the noise $\Falpha$.

For all $\alpha\in(0,2\wedge d)$, define the inner product $\langle \cdot,\cdot \rangle_{\alpha}$ and the norms $\|\cdot\|_\alpha$ and $\|\cdot\|_{(\alpha)}$ as follows: for any smooth functions $\phi,\phi_1,\phi_2\in \mathcal{D}(Q)$, set
\begin{align*}
\langle \phi_1,\phi_2 \rangle_\alpha&=\iint_{Q\times Q} \phi_1(x) \phi_2(y) |x-y|^{-\alpha} \diff x \diff y,\\
\|\phi\|^2_\alpha&=\iint_{Q\times Q} \phi(x) \phi(y) |x-y|^{-\alpha} 
\diff x \diff y,\\
\norm{\phi}^2_{(\alpha)}&=\iint_{Q\times Q} |\phi(x)| |\phi(y)| |x-y|^{-\alpha}
 \diff x \diff y.
\end{align*}
Let us introduce the Hilbert space $\mathcal{H}_{\alpha}$ as the completion of the space $\mathcal{D}(Q)$ with respect to the inner product $\langle \cdot,\cdot \rangle_\alpha$ and the associated norm $\|\cdot\|_\alpha$.
Any function $g:\overline{Q}\to \mathbb{R}$ satisfying
$\norm{g}^2_{(\alpha)}<\infty$ belongs to
$\mathcal{H}_\alpha$ (see, e.g.,\cite[Prop. 2.6(a)]{MR2785545}).

Next, define the Hilbert space $\mathbb{H}_{\alpha}=L^2([0,\infty);\mathcal{H}_\alpha)$ with inner product $\langle \cdot,\cdot\rangle_{\mathbb{H}_\alpha}$ and norm $\|\cdot\|_{\mathbb{H}_\alpha}$ defined as follows: for all $\varphi,\varphi_1,\varphi_2\in\mathbb{H}_\alpha$, set
\begin{align*}
\langle \varphi_1,\varphi_2\rangle_{\mathbb{H}_{\alpha}}&=\int_0^\infty \langle \varphi_1(t,\cdot),\varphi_2(t,\cdot)\rangle_\alpha  \diff t,\\
\|\varphi\|^2_{\mathbb{H}_{\alpha}}&=\int_0^\infty \|\varphi(t,\cdot)\|^2_\alpha  \diff t.
\end{align*}

Given the parameter $\alpha\in(0,2\wedge d)$, the Riesz noise $F_\alpha$ is defined as a family $\Falpha=\{\Falpha(\varphi); \, \varphi\in \mathbb{H}_\alpha\}$ of centered Gaussian random variables, characterized by the following covariance structure: for all $\varphi_1,\varphi_2\in\mathbb{H}_\alpha$, one has
\[
\E\left[\Falpha(\varphi_1)\Falpha(\varphi_2)\right]=\langle \varphi_1,\varphi_2\rangle_{\mathbb{H}_{\alpha}}.
\]
For any time $s\geq 0$ and any Borel set $A\in\mathcal{B}(Q)$, the function $\varphi_{s,A}:(t,x)\in\R^+\times Q\mapsto \mathds{1}_{[0,s]}(t)\mathds{1}_A(x)$ belongs to $\mathbb{H}_{\alpha}$. Set
\[
\Falpha\left([0,s]\times A\right)=\Falpha\left(\varphi_{s,A}\right).
\]
For any $t\geq0$, we define $\mathcal F_t$ as the $\sigma$-algebra
\[
\mathcal{F}_t=\sigma\left(\left\{\Falpha([0,s]\times A);\,0\leq s\leq t, A\in \mathcal{B}(Q)\right\}\right).
\]
In the sequel, we consider predictable processes with respect to the filtration $\{\mathcal{F}_t\}_{t\geq 0}$.

The stochastic integral with respect to the noise $\Falpha$ is understood in the Walsh--Dalang sense
(see e.g. \cite[Chapter 2]{MR876085}, \cite[Chapter 2]{dalangSPDE} or \cite{Dalang99,MR2785545}).
Namely, if $\{X(s,y),\, s\geq 0, y\in \overline{Q}\}$ is any predictable random field
satisfying that $\E[\|X\|^2_{\mathbb{H}_\alpha}]<\infty$, then the stochastic integral $\int_0^\infty \int_Q X(s,y)\Falpha(\diff s,\diff y)$ is a well-defined centered random variable, which satisfies the following isometry property:
\begin{equation}\label{eq:iso1}
 \E\left[\left|\int_0^\infty\int_Q X(s,y)\Falpha(\diff s,\diff y)\right|^2\right]=
 \E[\|X\|^2_{\mathbb{H}_\alpha}].
\end{equation}
If $X$, in addition, is such that $\E\left[ \int_0^\infty \|X(s,\cdot)\|^2_{(\alpha)} \diff s\right]<\infty$, then one has
\begin{equation}\label{eq:iso2}
 \E\left[\left|\int_0^\infty\int_Q X(s,y)\Falpha(\diff s,\diff y)\right|^2\right]\leq \E\left[ \int_0^\infty \|X(s,\cdot)\|^2_{(\alpha)} \diff s\right].
\end{equation}
For all $\varphi\in \mathbb{H}_\alpha$, one has
\[
\int_0^\infty\int_Q \varphi(s,y)\Falpha(\diff s,\diff y)=F_\alpha(\varphi).
\]
Finally, let us recall the Burkholder--Davis--Gundy inequality, see e.g. \cite[Appendix B]{MR3222416}: for all $p \geq 1$ and $T \in (0,\infty)$, there exists $C_{p}(T)\in(0,\infty)$ such that for any predictable process $X$ such that $\int_0^T \|X(s,\cdot)\|^2_{(\alpha)} \diff s\in L^p(\Omega)$, the stochastic integral satisfies $\int_0^T\int_Q X(s,y)\Falpha(\diff s,\diff y)\in L^{2p}(\Omega)$, and one has
\begin{equation}\label{bdg}
\E\left[\left|\int_0^T\int_Q X(s,y)\Falpha(\diff s,\diff y)\right|^{2p}\right]\leq C_{p}(T)\E\left[ \left(\int_0^T \|X(s,\cdot)\|^2_{(\alpha)} \diff s \right)^p\right],
\end{equation}
which can be rewritten as
\[
\vvvert \int_0^T\int_Q X(s,y)\Falpha(\diff s,\diff y) \vvvert_{2p}^{2p}\le C_p(T)\vvvert \int_0^T \|X(s,\cdot)\|^2_{(\alpha)} \diff s \vvvert_{p}^p.
\]

\begin{remark}
The Riesz noise $F_\alpha$ is indeed well-defined for all $\alpha\in (0,d)$. The restriction $\alpha<2$ is due to the fact that we want the following stochastic convolution to be well-defined, for all $T>0$:
\[
 \int_0^t\int_Q G_d(t-s,x,y)\Falpha(\diff s,\diff y),\; (t,x)\in [0,T]\times \overline{Q},
\]
where $G_d$ is the heat kernel introduced in the next Section \ref{sec:kernel}. More precisely, condition $\alpha <2$ is used to prove that
the function $(s,y)\in [0,t]\times \overline{O} \mapsto G_d(t-s,x,y)$ belongs to $\mathbb{H}_\alpha$ (see \eqref{eq:lem-aux1} in Lemma  \ref{lem:aux}). This is important in order to define the mild solution to equation \eqref{probIntro} (see \eqref{mild}).
\end{remark}

\subsection{Heat kernel}\label{sec:kernel}
In this subsection, we introduce the heat kernel $G_d$ on $Q$ with Dirichlet homogeneous boundary conditions, and provide several of its fundamental properties. For all $t\in(0,\infty)$ and all $x,y\in \overline{Q}$, set
\[\label{green}
G_d(t,x,y)=\sum_{\uk\in\N^d}\exp(-|\uk|^2\pi^2t)\varphi_{\uk}(x)\varphi_{\uk}(y),
\]
where for all $\uk\in\N^d$, $|\uk|^2$ and $\varphi_{\uk}$ are defined in Subsection~\ref{sec:notation}.

Let us first recall standard properties of the heat kernel $G_d$. First, one has
\begin{equation}\label{eq:Gdpos}
G_d(t,x,y)\ge 0,\qquad \forall t\in(0,\infty),~x,y\in \overline{Q}.
\end{equation}
Moreover, for all $t\in(0,\infty)$ and $x\in\overline{Q}$, one has
\begin{equation}\label{eq:Gdint}
\int_Q G_d(t,x,y) \diff y\in[0,1].
\end{equation}

Next, let us provide estimates for the heat kernel and its first order spatial and temporal derivatives.

\begin{lemma}\label{lem:GdEstimates}
The heat kernel $G_d$, its spatial gradient $\nabla_{x} G_{d}$ and its temporal derivative $\partial_{t} G_{d}$ satisfy the following upper bounds:
there exist $c_{d},C_{d}\in(0,\infty)$ such that for all $(t,x,y) \in (0,\infty) \times \overline{Q} \times \overline{Q}$, one has
\begin{align}
G_d(t,x,y)&\le C_dt^{-\frac{d}{2}}e^{-c_d\frac{|y-x|^2}{t}},\label{eq:GdEstimates-G}\\
|\nabla_{x} G_d(t,x,y)|&\le C_dt^{-\frac{d+1}{2}}e^{-c_d\frac{|y-x|^2}{t}},\label{eq:GdEstimates-dxG}\\
|\partial_tG_d(t,x,y)|&\le C_dt^{-\frac{d+2}{2}}e^{-c_d\frac{|y-x|^2}{t}}.\label{eq:GdEstimates-dtG}
\end{align}
\end{lemma}
We refer for instance to~\cite[page $274$ in Chapter 4]{Ladyenskaja1968LinearAQ} for a proof of Lemma~\ref{lem:GdEstimates}.
Combining \eqref{eq:GdEstimates-G} and \eqref{eq:GdEstimates-dtG} yields the following inequality.
\begin{lemma}\label{lem:GdEstimates-interpoled}
There exists $C_d\in(0,\infty)$ such that for all $t_1,t_2\in(0,\infty)$ such that $t_1<t_2$, and all $x,y\in \overline{Q}$, one has
\begin{equation}\label{eq:GdEstimates-interpoled-strong}
\begin{aligned}
\big|G_d(t_2,x,y)-G_d(t_1,x,y)\big|&\le C_d \mathds{1}_{\frac{t_2-t_1}{t_1}< 1}\frac{t_2-t_1}{t_1}t_1^{-\frac{d}{2}}e^{-c_d\frac{|y-x|^2}{2t_1}}\\
&+C_d \mathds{1}_{\frac{t_2-t_1}{t_1}\ge 1}\Bigl(t_1^{-\frac{d}{2}}e^{-c_d\frac{|y-x|^2}{2t_1}}+t_2^{-\frac{d}{2}}e^{-c_d\frac{|y-x|^2}{2t_2}}\Bigr).
\end{aligned}
\end{equation}
In addition, for all $\gamma\in[0,1]$, for all $t_1,t_2\in(0,\infty)$ such that $t_1<t_2$, and all $x,y\in \overline{Q}$, one has
\begin{equation}\label{eq:GdEstimates-interpoled}
\big|G_d(t_2,x,y)-G_d(t_1,x,y)\big|\le C_{d}\frac{|t_2-t_1|^\gamma}{t_1^\gamma}\Bigl(t_1^{-\frac{d}{2}}e^{-c_d\frac{|y-x|^2}{2t_1}}+t_2^{-\frac{d}{2}}e^{-c_d\frac{|y-x|^2}{2t_2}}\Bigr).
\end{equation}
\end{lemma}

Let us state additional properties of the heat kernel $G_d$, which are not standard since they are concerned with the norms $\|\cdot\|_{(\alpha)}$. Some of the estimates below can be found in \cite[Appendix~A]{MR2147242}. For the sake of completeness, we add the whole proof of the lemma below in Appendix \ref{app:aux}.

\begin{lemma}\label{lem:aux}
Let $d\in\N$ and $\alpha\in(0,2\wedge d)$.
There exists $C_\alpha\in(0,\infty)$ such that the following inequalities hold:
\begin{align}
&\underset{x\in \overline{Q}}\sup~\|G_d(t,x,\cdot)\|_{(\alpha)}^2\le C_\alpha t^{-\frac{\alpha}{2}}, \quad \forall~t\in(0,\infty),\label{eq:lem-aux1}\\
&\underset{x\in\overline{Q} }\sup~\int_{t_1}^{t_2}\|G_d(t_2-t,x,\cdot)\|_{(\alpha)}^2 \diff t\le C_\alpha(T)|t_2-t_1|^{1-\frac{\alpha}{2}}, \quad \forall~t_2\ge t_1\ge 0,\label{eq:lem-aux2}\\
&\int_{0}^{+\infty}\|G_d(t,x_2,\cdot)-G_d(t,x_1,\cdot)\|_{(\alpha)}^{2}dt\le C_\alpha|x_2-x_1|^{2-\alpha},\quad \forall~x_1,x_2\in \overline{Q}.\label{eq:lem-aux3}
\end{align}

Moreover, for all $\gamma\in(0,\frac12-\frac{\alpha}{4})$, there exists $C_{\alpha,\gamma}\in(0,\infty)$ such that
\begin{equation}\label{eq:lem-aux4}
\underset{x\in \overline{Q} }\sup~\int_{0}^{t_1}\|G_d(t_2-t,x,\cdot)-G_d(t_1-t,x,\cdot)\|_{(\alpha)}^{2} \diff t\le C_{\alpha,\gamma} t_1^{1-\frac{\alpha}{2}-2\gamma}|t_2-t_1|^{2\gamma},\quad \forall~t_2\ge t_1\ge 0.
\end{equation}
\end{lemma}

Let us finally provide a technical lemma on estimates for stochastic integrals of the heat kernel $G_d$ with respect to the noise $\Falpha$.

\begin{lemma}\label{lem:tech}
Let $d\in\N$, $\alpha\in(0,2\wedge d)$, $T\in(0,\infty)$, and $p\in[1,\infty)$. There exists $C_{p,\alpha}(T)\in(0,\infty)$ such that the following holds: if $\kappa\colon\mathbb R_+\to \mathbb R_+$ is a mapping such that $\kappa(s)\leq s$ for all $s\in\mathbb R_+$, and if $\{X(s,x),\, (s,x)\in [0,T]\times \overline{Q}\}$ is any square-integrable predictable process taking values in $\mathbb{H}_\alpha$, then one has the upper bounds
\begin{align}
\underset{0\le t\le T}\sup~\underset{x\in \overline{Q}}\sup~\vvvert & \int_{0}^{t}\int_Q G_d(t-\kappa(s),x,y)X(s,y)\Falpha(\diff s,\diff y)\vvvert_{2p}\nonumber\\
&\le C_{p,\alpha}(T)\left(\int_0^t (t-\kappa(s))^{-\frac{\alpha}{2}}\underset{x\in \overline{Q}}\sup~\vvvert X(s,x)\vvvert_{2p}^2\diff s\right)^{\frac12}\label{eq:lem-tech1}\\
&\le C_{p,\alpha}(T)\left(\int_0^t (t-\kappa(s))^{-\frac{\alpha}{2}}\underset{x\in \overline{Q}}\sup~\E[|X(s,x)|^{2p}]\diff s\right)^{\frac1{2p}}\label{eq:lem-tech2}\\
&\le C_{p,\alpha}(T)\underset{0\le t\le T}\sup~\underset{x\in \overline{Q}}\sup~\vvvert X(t,x)\vvvert_{2p}\label{eq:lem-tech3}.
\end{align}
\end{lemma}

The proofs of Lemma~\ref{lem:GdEstimates-interpoled}, of Lemma~\ref{lem:aux} and of Lemma~\ref{lem:tech} are postponed to Appendix~\ref{app:aux}.

\subsection{Regularity assumptions}\label{sec:ass}

We conclude this section by giving the assumptions on the initial value $u_0$ and on the coefficients $b$ and $\sigma$ which appear in
the stochastic evolution equation.

\begin{assumption}\label{ass:init}
The initial value $u_0$ is a function such that $u_0\in\mathcal{C}^{0}_{0}(\overline{Q})$. In other words, it is assumed that $u_0\colon \overline{Q}\to\R$ is continuous, and that $u_0(x)=0$ for all $x\in\partial Q=\overline{Q}\setminus Q$.
\end{assumption}
Note that in this article the initial value $u_0$ is assumed to be deterministic. Considering $\mathcal{F}_0$-measurable random initial values, independent of the Riesz noise $F_\alpha$, would be possible up to minor standard modifications.

Let $T\in(0,\infty)$ denote a fixed time. The mappings $b$ and $\sigma$ are assumed to satisfy the following conditions.
\begin{assumption}\label{ass:coeff}
The mappings $b\colon[0,T]\times
\overline{Q}\times\R\to\R$ and $\sigma\colon[0,T]\times \overline{Q}\times\R\to\R$ satisfy the following assumptions: there exist $C\in(0,\infty)$ such that,
for all $s,t\in[0,T]$, $x,y\in \overline{Q}$, $u,v\in\R$, one has
\begin{align}\label{L}
|b(s,x,u)-b(t,y,v)|+|\sigma(s,x,u)-\sigma(t,y,v)|\leq
C\bigl(|s-t|^{1/2-\alpha/4}+|x-y|^{1-\alpha/2}+|u-v|\bigr),\tag{L}
\end{align}
and such that, for all $t\in[0,T]$, $x\in \overline{Q}$, $u\in\R$, one has
\begin{align}\label{LG}
|b(t,x,u)|+|\sigma(t,x,u)|\leq C(1+|u|).\tag{LG}
\end{align}
\end{assumption}

Observe that the condition~\eqref{L} means that the two mappings $u\in\R\mapsto b(t,x,u)$ and $u\in\R\mapsto \sigma(t,x,u)$ are globally Lipschitz continuous, uniformly with respect to $t\in[0,T]$ and $x\in \overline{Q}$. Moreover, the temporal and spatial H\"older regularity exponents
$1/2-\alpha/4$ and $1-\alpha/2$ in the condition~\eqref{L} are related to the optimal temporal and spatial regularity properties of the solutions, see Section~\ref{sec-exact} and e.g. \cite[Prop. 3.2]{MR2147242}.

\section{Well-posedness and properties of the SPDE driven by Riesz noise}\label{sec-exact}

The objective of this section is to study the stochastic partial differential equation~\eqref{probIntro} driven by  Riesz noise in dimension $d$, employing the tools presented in Section~\ref{sec-sett}. Proposition~\ref{prop:uExistence} provides existence and uniqueness of a mild solution as well as moment bounds. Proposition~\ref{prop:reg} provides a temporal regularity result.

We consider the stochastic partial differential equation
\begin{equation}\label{prob}
\left\lbrace
\begin{aligned}
&\frac{\partial u}{\partial t}(t,x)=\Delta u(t,x)+b(t,x,u(t,x))+\sigma(t,x,u(t,x))\frac{\partial^2\Falpha}{\partial t\partial x}(t,x),\quad \forall~t>0,x\in Q,\\
&u(t,x)=0,\quad \forall~x\in\partial Q,\\
&u(0,x)=u_0(x),\quad \forall~x\in \overline{Q},
\end{aligned}
\right.
\end{equation}
where the initial value $u_{0}$ satisfies Assumption~\ref{ass:init} and the mappings $b$ and $\sigma$ satisfy Assumption~\ref{ass:coeff}. The formulation~\eqref{prob} is formal. We recall that a mild solution to~\eqref{prob} is an adapted process $\{u(t,x),\,
(t,x)\in [0,T]\times \overline{Q}\}$ which satisfies the following identity: almost surely, for any $(t,x)\in (0,T]\times \overline{Q}$, one has
\begin{align}\label{mild}
u(t,x)&=\int_QG_d(t,x,y)u_0(y)\diff y+\int_0^t\int_QG_d(t-s,x,y)b(s,y,u(s,y))\diff y \diff s\nonumber\\
&\quad+\int_0^t\int_QG_d(t-s,x,y)\sigma(s,y,u(s,y))\Falpha(\diff s,\diff
y).
\end{align}
The stochastic integral with respect to the Riesz noise $\Falpha(\diff s,\diff y)$ appearing in~\eqref{mild} is understood in the Walsh--Dalang sense, as presented in Subsection~\ref{sec:noise}.

Let us state properties of mild solutions of the parabolic SPDE~\eqref{prob}.
\begin{proposition}\label{prop:uExistence}
Let Assumptions~\ref{ass:init} and~\ref{ass:coeff} be satisfied. For all $T\in(0,\infty)$, there exists a unique global mild solution $\{u(t,x),\,(t,x)\in [0,T]\times \overline{Q}\}$ to~\eqref{prob}.

In addition, the mild solution satisfies the following moment bounds: for all $p\in[1,\infty)$ and all $T\in(0,\infty)$, there exists $C_{p}(T)\in(0,\infty)$ such that one has
\begin{equation}\label{eq:momo-u}
\underset{t\in[0,T]}\sup~\underset{x\in \overline{Q}}\sup~\E[|u(t,x)|^{2p}]\le C_p(T)\bigl(1+\|u_0\|_\infty^{2p}\bigr).
\end{equation}
\end{proposition}

\begin{proof}
Existence of the unique mild solution $u(t,x)$ to the parabolic SPDE~\eqref{prob} can be found in, for example,
\cite{Dalang99,MR2054575,MR4346865,dalangSPDE}.

The moment bounds of $u(t,x)$ are established using standard arguments. A detailed proof is given for completeness.

For all $t\in[0,T]$ and all $x\in Q$, one has
\begin{equation}\label{momo}
\E[|u(t,x)|^{2p}]\le 3^{2p-1}\Bigl(M_p^1(t,x)+M_p^2(t,x)+M_p^3(t,x)\Bigr),
\end{equation}
where $M_p^1(t,x)$, $M_p^2(t,x)$ and $M_p^3(t,x)$ are defined as
\begin{align*}
M_p^1(t,x)&=\left|\int_QG_d(t,x,y)u_0(y)\diff y\right|^{2p}\nonumber\\
M_p^2(t,x)&=\E\left[\left|\int_0^t\int_QG_d(t-s,x,y)b(s,y,u(s,y))\diff y \diff s\right|^{2p}\right]\nonumber\\
M_p^3(t,x)&=\E\left[\left|\int_0^t\int_QG_d(t-s,x,y)\sigma(s,y,u(s,y))\Falpha(\diff s,\diff y)\right|^{2p}\right].
\end{align*}

First, owing to the properties~\eqref{eq:Gdpos} and~\eqref{eq:Gdint} of the heat kernel $G_d$ (see Subsection~\ref{sec:kernel}), one has 
\[
\left|\int_Q G_d(t,x,y)u_0(y)\diff y\right|\le \int_Q G_d(t,x,y)|u_0(y)|\diff y\le \underset{y\in Q}\sup~|u_0(y)| \int_Q G_d(t,x,y)\diff y\le \|u_0\|_\infty.
\]
As a result, one obtains
\begin{equation}\label{momo1}
\underset{t\in[0,T]}\sup~\underset{x\in \overline{Q}}\sup~M_p^1(t,x)\le \|u_0\|_\infty^{2p}.
\end{equation}
Let us then consider $M_p^2(t,x)$. Applying the H\"older inequality and the properties~\eqref{eq:Gdpos} and~\eqref{eq:Gdint} of the heat kernel $G_d$ (see Subsection~\ref{sec:kernel}), one obtains
\begin{align*}
M_p^2(t,x)&\le \int_0^t\int_QG_d(t-s,x,y)\E[|b(s,y,u(s,y))|^{2p}]\diff y \diff s \left(\int_0^t\int_QG_d(t-s,x,y)\diff y \diff s\right)^{2p-1}\\
&\le \int_0^t\underset{y\in \overline{Q}}\sup~\E[|b(s,y,u(s,y))|^{2p}]\diff s ~t^{2p-1}.
\end{align*}
Owing to~\eqref{LG}, the mapping $b$ has at most linear growth, therefore one has
\[
\underset{y\in \overline{Q}}\sup~\E[|b(s,y,u(s,y))|^{2p}]\le C_p\bigl(1+\underset{y\in \overline{Q}}\sup~\E[|u(s,y)|^{2p}]\bigr).
\]
As a result, one obtains
\begin{equation}\label{momo2}
M_p^2(t,x)\le C_{p}(T)\int_0^t \bigl(1+\underset{{y}\in \overline{Q}}\sup~\E[|u(s,y)|^{2p}]\bigr)\diff s.
\end{equation}

Finally, let us consider $M_p^3(t,x)$. Applying the inequality~\eqref{eq:lem-tech2} from Lemma~\ref{lem:tech} (with $\kappa(s)=s$ for all $s\ge 0$), one obtains
\[
M_p^3(t,x)\le C_{p,\alpha}(T)\int_0^t (t-s)^{-\frac{\alpha}{2}}\underset{{y}\in \overline{Q}}\sup~\E[|\sigma(s,y,u(s,y))|^{2p}]\diff s.
\]
Owing to~\eqref{LG}, the mapping $\sigma$ has at most linear growth, therefore one has
\[
\underset{y\in \overline{Q}}\sup~\E[|\sigma(s,y,u(s,y))|^{2p}]\le C_p\bigl(1+\underset{y\in \overline{Q}}\sup~\E[|u(s,y)|^{2p}]\bigr).
\]
Hence, one gets that
\begin{equation}\label{momo3}
M_p^3(t,x)\le C_{p,\alpha}(T)\int_0^t (t-s)^{-\frac{\alpha}{2}}\bigl(1+\underset{{y}\in \overline{Q}}\sup~\E[|u(s,y)|^{2p}]\bigr)\diff s.
\end{equation}
Combining the upper bounds~\eqref{momo1},~\eqref{momo2} and~\eqref{momo3}, one obtains the following inequality: there exists $C_{p,\alpha}(T)\in(0,\infty)$ such that for all $t\in[0,T]$ one has
\[
\underset{x\in \overline{Q}}\sup~\E[|u(t,x)|^{2p}]\le C_{p,\alpha}(T)\left(\|u_0\|_\infty^{2p}+\int_0^t \left( 1 + (t-s)^{-\frac{\alpha}{2}} \right)\bigl(1+\underset{{y}\in \overline{Q}}\sup~\E[|u(s,y)|^{2p}]\bigr)\diff s\right).
\]
Applying a Gr\"onwall inequality (e.g., \cite[Lem. 15]{Dalang99}) then provides the moment bounds~\eqref{eq:momo-u}.

This concludes the proof of Proposition~\ref{prop:uExistence}.
\end{proof}

Let us now state and prove a result on temporal H\"older regularity of the mild solution~\eqref{mild} to~\eqref{prob}. The version presented in Proposition~\ref{prop:reg} has not previously been obtained in the literature, to the best of our knowledge.

\begin{proposition}\label{prop:reg}
Let Assumptions~\ref{ass:init} and~\ref{ass:coeff} be satisfied. Let $T\in(0,\infty)$ and $\{u(t,x),\,(t,x)\in [0,T]\times \overline{Q}\}$ be the unique mild solution to~\eqref{prob}.
For all $\alpha\in(0,2\wedge d)$, $p\in[1,\infty)$ and $\gamma\in(0,\frac12-\frac{\alpha}{4})$, there exists $C_{p,\alpha,\gamma}(T)\in(0,\infty)$ such that for all $0<t_1\leq t_2\le T$ one has
\begin{equation}\label{eq:reg}
\underset{x\in \overline{Q}}\sup~\bigl(\E[|u(t_2,x)-u(t_1,x)|^{p}]\bigr)^{\frac{1}{p}}\le C_{p,\alpha,\gamma}(T)\frac{|t_2-t_1|^\gamma}{t_1^\gamma}(1+\|u_0\|_\infty).
\end{equation}
\end{proposition}

\begin{proof}
Assume that $0<t_1 \le t_2 \le T$ and let $x\in Q$, then one has the decomposition
\[
u(t_2,x)-u(t_1,x)=\delta^1(t_1,t_2,x)+\delta^2(t_1,t_2,x)+\delta^3(t_1,t_2,x)+\delta^4(t_1,t_2,x)+\delta^5(t_1,t_2,x),
\]
where the error terms $\delta^j(t_1,t_2,x)$ for all $j\in\{1,\ldots,5\}$ are given by
\begin{align*}
\delta^1(t_1,t_2,x)&=\int_Q\left[G_d(t_2,x,y)-G_d(t_1,x,y)\right]u_0(y)\diff y\\
\delta^2(t_1,t_2,x)&=\int_{0}^{t_1}\int_Q\left[G_d(t_2-s,x,y)-G_d(t_1-s,x,y)\right]b(s,y,u(s,y))\diff y \diff s\\
\delta^3(t_1,t_2,x)&=\int_{t_1}^{t_2}\int_QG_d(t_2-s,x,y)b(s,y,u(s,y))\diff y \diff s\\
\delta^4(t_1,t_2,x)&=\int_{0}^{t_1}\int_Q\left[G_d(t_2-s,x,y)-G_d(t_1-s,x,y)\right]\sigma(s,y,u(s,y))\Falpha(\diff s,\diff y)\\
\delta^5(t_1,t_2,x)&=\int_{t_1}^{t_2}\int_QG_d(t_2-s,x,y)\sigma(s,y,u(s,y))\Falpha(\diff s,\diff y).
\end{align*}

Let us consider the term $\delta^1(t_1,t_2,x)$. Assume that $\gamma\in(0,1)$. Applying the inequality~\eqref{eq:GdEstimates-interpoled} on the heat kernel (see Subsection~\ref{sec:kernel}), one obtains the following upper bounds:
\begin{align*}
\Big|\int_Q\left[G_d(t_2,x,y)-G_d(t_1,x,y)\right]u_0(y)\diff y \Big| & \le \int_Q\big|G_d(t_2,x,y)-G_d(t_1,x,y)\big| |u_0(y)|\diff y \\
& \le \underset{y\in \overline{Q}}\sup~|u_0(y)| \int_Q\big|G_d(t_2,x,y)-G_d(t_1,x,y)\big| \diff y \\
& \le C \frac{|t_2-t_1|^{\gamma}}{t_1^\gamma}\|u_0\|_\infty.
\end{align*}
Therefore one obtains the following upper bound for $\delta^1(t_1,t_2,x)$: for all $\gamma\in(0,1)$, there exists $C_\gamma\in(0,\infty)$ such that
\begin{equation}\label{eq:delta1}
\underset{x\in \overline{Q}}\sup~|\delta^1(t_1,t_2,x)|\le C_\gamma \frac{|t_2-t_1|^{\gamma}}{t_1^\gamma}\|u_0\|_\infty,\quad \forall~0<t_1<t_2.
\end{equation}

Let us then consider the term $\delta^2(t_1,t_2,x)$. Let $p\in[1,\infty)$. Applying the Minkowski inequality, one has
\begin{align*}
\vvvert\int_{0}^{t_1}\int_Q & \left[G_d(t_2-s,x,y)-G_d(t_1-s,x,y)\right]b(s,y,u(s,y))\diff y \diff s\vvvert_p\\
&\le \int_{0}^{t_1}\int_Q \big|G_d(t_2-s,x,y)-G_d(t_1-s,x,y)\big|\vvvert b(s,y,u(s,y))\vvvert_p\diff y \diff s\\
&\le \underset{s\in[0,T]}\sup~\underset{y\in \overline{Q}}\sup~\vvvert b(s,y,u(s,y))\vvvert_p\int_{0}^{t_1}\int_Q \big|G_d(t_2-s,x,y)-G_d(t_1-s,x,y)\big|\diff y \diff s.
\end{align*}
Owing to the condition~\eqref{LG} from Assumption~\ref{ass:coeff}, the mapping $b$ has at most linear growth. Therefore, applying the moment bounds~\eqref{eq:momo-u} from Proposition~\eqref{prop:uExistence}, one obtains the upper bound
\[
\underset{s\in[0,T]}\sup~\underset{y\in \overline{Q}}\sup~\vvvert b(s,y,u(s,y))\vvvert_p\le C\bigl(1+\underset{s\in[0,T]}\sup~\underset{y\in \overline{Q}}\sup~\vvvert u(s,y)\vvvert_p\bigr)\le C_p(T)\bigl(1+\|u_0\|_\infty\bigr).
\]
Applying the inequality~\eqref{eq:GdEstimates-interpoled} on the heat kernel (see Subsection~\ref{sec:kernel}), one obtains
\begin{align*}
\int_{0}^{t_1}\int_Q \big|G_d(t_2-s,x,y)-G_d(t_1-s,x,y)\big|\diff y \diff s &\le C \int_{0}^{t_1}\frac{|t_2-t_1|^\gamma}{(t_1-s)^\gamma}\int_Q (t_1-s)^{-\frac{d}{2}}e^{-c_{d} \frac{|y-x|^2}{2(t_1-s)}} \diff y  \diff s\\
&+C \int_{0}^{t_1}\frac{|t_2-t_1|^\gamma}{(t_1-s)^\gamma}\int_Q (t_2-s)^{-\frac{d}{2}}e^{-c_{d} \frac{|y-x|^2}{2(t_2-s)}} \diff y  \diff s\\
&\le C \int_{0}^{t_1}\int_Q \frac{|t_2-t_1|^\gamma}{(t_1-s)^\gamma} \diff s\\
&\le C_\gamma(T)|t_2-t_1|^\gamma.
\end{align*}
Therefore one obtains the following upper bound for $\delta^2(t_1,t_2,x)$: for all $p\in[1,\infty)$ and $\gamma\in(0,1)$, there exists $C_{p,\gamma}(T)\in(0,\infty)$ such that
\begin{equation}\label{eq:delta2}
\underset{x\in \overline{Q}}\sup~\vvvert\delta^2(t_1,t_2,x)\vvvert_p\le C_{p,\gamma}(T) |t_2-t_1|^{\gamma}\bigl(1+\|u_0\|_\infty\bigr),\quad \forall~0<t_1<t_2\le T.
\end{equation}

The treatment of the term $\delta^3(t_1,t_2,x)$ is similar: applying the Minkowski inequality and the property~\eqref{eq:Gdint} of the heat kernel (see Subsection~\ref{sec:kernel}), one has
\begin{align*}
\vvvert \int_{t_1}^{t_2}\int_QG_d(t_2-s,x,y)b(s,y,u(s,y))\diff y \diff s \vvvert_p &\le \int_{t_1}^{t_2}\int_QG_d(t_2-s,x,y)\vvvert b(s,y,u(s,y))\vvvert_p\diff y \diff s\\
&\le \underset{s\in[0,T]}\sup~\underset{y\in \overline{Q}}\sup~\vvvert b(s,y,u(s,y))\vvvert_p \int_{t_1}^{t_2}\int_QG_d(t_2-s,x,y)\diff y \diff s \\
&\le C_p(T)\bigl(1+\|u_0\|_\infty\bigr)|t_2-t_1|.
\end{align*}
Therefore one obtains the following upper bound for $\delta^3(t_1,t_2,x)$: for all $p\in[1,\infty)$ and $\gamma\in(0,1)$, there exists $C_{p,\gamma}(T)\in(0,\infty)$ such that
\begin{equation}\label{eq:delta3}
\underset{x\in \overline{Q}}\sup~\vvvert\delta^3(t_1,t_2,x)\vvvert_p\le C_{p,\gamma}(T) |t_2-t_1|^{\gamma}\bigl(1+\|u_0\|_\infty\bigr),\quad \forall~0<t_1<t_2\le T.
\end{equation}

It remains to deal with the terms $\delta^4(t_1,t_2,x)$ and $\delta^5(t_1,t_2,x)$. Applying the Burkholder--Davis--Gundy inequality~\eqref{bdg}, one obtains the upper bound
\begin{align*}
\E[|\delta^4(t_1,t_2,x)|^{2p}]&=\E \left[ \left|\int_{0}^{t_1}\int_Q\left[G_d(t_2-s,x,y)-G_d(t_1-s,x,y)\right]\sigma(s,y,u(s,y))\Falpha(\diff s,\diff y) \right|^{2p} \right]\\
&\le C_p(T)\E \left[ \left|\int_{0}^{t_1}\|\left[G_d(t_2-s,x,\cdot)-G_d(t_1-s,x,\cdot)\right]\sigma(s,\cdot,u(s,\cdot))\|_{(\alpha)}^2 \right|^p \right].
\end{align*}
For all $(t_1,t_2,x)\in [0,T]^2\times \overline{Q}$ (with $t_1<t_2$), introduce the auxiliary positive measure $\nu_{t_1,t_2,x}^{(1)}$ on $[0,t_1]\times Q^2$, which is absolutely continuous with respect to the Lebesgue measure, and with Radon--Nikodym derivative given by
\begin{align*}
\frac{\diff\nu_{t_1,t_2,x}^{(1)}(s,y,z)}{\diff s \diff y \diff z}=
\left|G_d(t_2-s,x,y)-G_d(t_1-s,x,y)\right|\left|G_d(t_2-s,x,z)-G_d(t_1-s,x,z)\right||y-z|^{-\alpha}.
\end{align*}
Making use of the auxiliary measure $\nu_{t_1,t_2,x}^{(1)}$ introduced above yields the upper bound
\[
\E[|\delta^4(t_1,t_2,x)|^{2p}]\le C_p(T)\E \left[ \left|\int_{0}^{t_1}\iint_{Q\times Q}|\sigma(s,y,u(s,y))| |\sigma(s,z,u(s,z))| \diff \nu_{t_1,t_2,x}^{(1)}(s,y,z) \right|^p \right].
\] 
Applying the H\"older inequality, one then obtains
\begin{align*}
\E[|\delta^4(t_1,t_2,x)|^{2p}]&\le \int_{0}^{t_1}\iint_{Q\times Q}\E[|\sigma(s,y,u(s,y))|^p |\sigma(s,z,u(s,z))|^p] \diff \nu_{t_1,t_2,x}^{(1)}(s,y,z)~ \nu_{t_1,t_2,x}^{(1)}([0,t_1]\times Q^2)^{p-1}\\
&\le \underset{s\in[0,T]}\sup~\underset{y,z\in \overline{Q}}\sup~\E[|\sigma(s,y,u(s,y))|^p |\sigma(s,z,u(s,z))|^p]~\nu_{t_1,t_2,x}^{(1)}([0,t_1]\times Q^2)^{p}.
\end{align*}
Owing to the condition~\eqref{LG} from Assumption~\ref{ass:coeff}, the mapping $\sigma$ has at most linear growth. Therefore, applying the moment bounds~\eqref{eq:momo-u} from Proposition~\eqref{prop:uExistence}, one obtains the upper bounds
\begin{align*}
\underset{s\in[0,T]}\sup~\underset{y,z\in \overline{Q}}\sup~\E[|\sigma(s,y,u(s,y))|^p |\sigma(s,z,u(s,z))|^p]&\le \underset{s\in[0,T]}\sup~\underset{y\in \overline{Q}}\sup~\E[|\sigma(s,y,u(s,y))|^{2p}]\\
&\le C_p(T)\bigl(1+\|u_0\|^{2p}_{\infty}\bigr).
\end{align*}
Moreover, applying the inequality~\eqref{eq:lem-aux4} from Lemma~\ref{lem:aux}, with the condition $\gamma\in(0,\frac12-\frac{\alpha}{4})$, if $0<t_1<t_2\le T$, one obtains the upper bound
\[
\nu_{t_1,t_2,x}^{(1)}([0,t_1]\times Q^2)=\int_{0}^{t_1}\|G_d(t_2-s,x,\cdot)-G_d(t_1-s,x,\cdot)\|_{(\alpha)}^2 \diff s \le C_{\alpha,\mu}(T)|t_2-t_1|^{2\gamma}.
\]
Gathering the upper bounds, one obtains the following upper bound for $\delta^4(t_1,t_2,x)$: for all $p\in[1,\infty)$ and $\gamma\in(0,\frac{1}{2}-\frac{\alpha}{4})$, there exists $C_{p,\alpha,\gamma}(T)\in(0,\infty)$ such that
\begin{equation}\label{eq:delta4}
\underset{x\in \overline{Q}}\sup~\vvvert\delta^4(t_1,t_2,x)\vvvert_{2p}\le C_{p,\alpha,\gamma}(T) |t_2-t_1|^{\gamma}\bigl(1+\|u_0\|_\infty\bigr),\quad \forall~0<t_1<t_2\le T.
\end{equation}

Finally, the treatment of the term $\delta^5(t_1,t_2,x)$ is similar to the treatment of $\delta^4(t_1,t_2,x)$ above. Applying the Burkholder--Davis--Gundy inequality~\eqref{bdg}, one obtains the upper bound
\begin{align*}
\E[|\delta^5(t_1,t_2,x)|^{2p}]=&\E \left[ \left|\int_{t_1}^{t_2}\int_QG_d(t_2-s,x,y)\sigma(s,y,u(s,y))\Falpha(\diff s,\diff y) \right|^{2p} \right]\\
&\le C_p(T)\E \left[ \left|\int_{t_1}^{t_2} \|G_d(t_2-s,x,\cdot)\sigma(s,\cdot,u(s,\cdot))\|_{(\alpha)}^2 \right|^p \right].
\end{align*}
For all $(t_1,t_2,x)\in [0,T]^2\times \overline{Q}$ (with $t_1<t_2$), we introduce the auxiliary positive measure $\nu_{t_1,t_2,x}^{(2)}$ on $[t_1,t_2]\times Q^2$, which is absolutely continuous with respect to the Lebesgue measure, and with Radon--Nikodym derivative given by
\[
\frac{\diff\nu_{t_1,t_2,x}^{(2)}(s,y,z)}{\diff s \diff y \diff z}=G_d(t_2-s,x,y)G_d(t_2-s,x,z)|y-z|^{-\alpha}.
\]
Making use of the auxiliary measure $\nu_{t_1,t_2,x}^{(2)}$ introduced above yields the upper bound
\[
\E[|\delta^5(t_1,t_2,x)|^{2p}]\le C_p(T)\E \left[ \left|\int_{t_1}^{t_2}\iint_{Q\times Q}|\sigma(s,y,u(s,y))| |\sigma(s,z,u(s,z))| \diff \nu_{t_1,t_2,x}^{(2)}(s,y,z) \right|^p \right].
\]
Applying the H\"older inequality, one then obtains
\begin{align*}
&\E \left[ \left|\int_{t_1}^{t_2}\iint_{Q\times Q}|\sigma(s,y,u(s,y))| |\sigma(s,z,u(s,z))| \diff \nu_{t_1,t_2,x}^{(2)}(s,y,z) \right|^p \right]\\
&\le \int_{t_1}^{t_2}\iint_{Q\times Q}\E[|\sigma(s,y,u(s,y))|^p |\sigma(s,z,u(s,z))|^p] \diff \nu_{t_1,t_2,x}^{(2)}(s,y,z)~ \nu_{t_1,t_2,x}^{(2)}([t_1,t_2]\times Q^2)^{p-1}\\
&\le \underset{s\in[0,T]}\sup~\underset{y,z\in \overline{Q}}\sup~\E[|\sigma(s,y,u(s,y))|^p |\sigma(s,z,u(s,z))|^p]~\nu_{t_1,t_2,x}^{(2)}([t_1,t_2]\times Q^2)^{p}.
\end{align*}
Owing to the condition~\eqref{LG} from Assumption~\ref{ass:coeff}, the mapping $\sigma$ has at most linear growth. Therefore, applying the moment bounds~\eqref{eq:momo-u} from Proposition~\eqref{prop:uExistence}, one obtains the upper bounds
\begin{align*}
\underset{s\in[0,T]}\sup~\underset{y,z\in \overline{Q}}\sup~\E[|\sigma(s,y,u(s,y))|^p |\sigma(s,z,u(s,z))|^p]&\le \underset{s\in[0,T]}\sup~\underset{y\in \overline{Q}}\sup~\E[|\sigma(s,y,u(s,y))|^{2p}]\\
&\le C_p(T)\bigl(1+\|u_0\|_{\infty}^{2p}\bigr).
\end{align*}
Moreover, applying the inequality~\eqref{eq:lem-aux4} from Lemma~\ref{lem:aux}, if $0<t_1<t_2\le T$, one obtains the upper bound
\[
\nu_{t_1,t_2,x}^{(2)}([t_1,t_2]\times Q^2)=\int_{t_1}^{t_2}\|G_d(t_2-s,x,\cdot)\|_{(\alpha)}^2 \diff s \le C_{\alpha}(T)|t_2-t_1|^{1-\frac{\alpha}{2}}.
\]
Gathering the upper bounds, and taking into account the condition $\gamma\in(0,\frac{1}{2}-\frac{\alpha}{4})$, one obtains the following upper bound for $\delta^5(t_1,t_2,x)$: for all $\gamma\in(0,\frac{1}{2}-\frac{\alpha}{4})$ and $p\in[1,\infty)$, there exists $C_{p,\alpha,\gamma}(T)\in(0,\infty)$ such that
\begin{equation}\label{eq:delta5}
\underset{x\in \overline{Q}}\sup~\vvvert\delta^5(t_1,t_2,x)\vvvert_{2p}\le C_{p,\alpha,\gamma}(T) |t_2-t_1|^{\gamma}\bigl(1+\|u_0\|_\infty\bigr),\quad \forall~0<t_1<t_2\le T.
\end{equation}
Combining the upper bounds~\eqref{eq:delta1},~\eqref{eq:delta2},~\eqref{eq:delta3},
~\eqref{eq:delta4} and~\eqref{eq:delta5} then provides the inequality~\eqref{eq:reg}. The proof of Proposition~\ref{prop:reg} is thus completed.
\end{proof}

\section{Convergence analysis of the stochastic exponential integrator}\label{sec-conv}

In this section, we introduce the stochastic exponential integrator given by~\eqref{timeapp} and~\eqref{sexp} applied to the stochastic partial differential equation~\eqref{prob}. We then state Theorem~\ref{thm:main}, which is the main result of this article, and provides strong convergence of the integrator when the time-step size $\dt$ goes to $0$, with rate $\gamma$ arbitrarily close to $\frac12-\frac{\alpha}{4}$.

Let $T\in(0,\infty)$ and given an integer $m\in\N$ define the time-step size $\dt=\frac{T}{m}$. For all $\ell\in\{0,\ldots,m\}$, introduce the grid times $t_\ell=\ell\dt$. In addition, let $\lfloor\cdot\rfloor$ denote the integer part, and for all $s\in[0,T]$ set
\[
\kappa_m^T(s)=\dt\left\lfloor\frac{s}{\dt}\right\rfloor=\frac{T}{m}\left\lfloor \frac{ms}{T}\right\rfloor=\sup~\{t_\ell;~t_\ell\le s\}.
\]

Let us first start by giving a continuous-time mild formulation of the stochastic exponential integrator which is inspired
by \cite{MR4050540} which treats the case of one dimensional SPDE driven by space-time white noise. The random field $\{u^{m}(t,x),\, (t,x)\in [0,T]\times \overline{Q}\}$ is defined as the solution to the following system: for all $t\in[0,T]$ and $x\in \overline{Q}$,
\begin{equation}\label{timeapp}
\begin{aligned}
u^{m}(t,x)&=\int_Q G_d(t,x,y)u_0(y)\diff y\\
&\quad+\int_0^t\int_Q G_d(t-\kappa_m^T(s),x,y)b(\kappa_m^T(s),y,u^{m}(\kappa_m^T(s),y))\diff y\diff s\\
&\quad+\int_0^t\int_Q G_d(t-\kappa_m^T(s),x,y)\sigma(\kappa_m^T(s),y,u^{m}(\kappa_m^T(s),y))\,\Falpha(\diff s,\diff y).
\end{aligned}
\end{equation}
The formulation~\eqref{timeapp} of the stochastic exponential integrator is mainly useful for the theoretical
analysis of the scheme.
For its practical implementation, it is sufficient to consider approximations $\mathcal U^m_\ell=u^m(t_\ell,\cdot)$ at the grid times $t_\ell$. From the definition~\eqref{timeapp} of the integrator and the semigroup property of the heat kernel $G_d$, for all $\ell\in\{0,\ldots,m-1\}$ and all $x\in \overline{Q}$, one has
\begin{equation}\label{sexp0}
\begin{aligned}
\mathcal{U}_{\ell+1}^{m}(x)&=\int_Q G_d(\dt,x,y)\mathcal{U}_{\ell}^{m}(y)\diff y\\
&\quad+\dt\int_Q G_d(\dt,x,y)b(t_\ell,y,\mathcal{U}_{\ell}^m(y))\diff y\\
&\quad+\int_{t_{\ell}}^{t_{\ell+1}}\int_Q G_d(\dt,x,y)\sigma(t_\ell,y,\mathcal{U}_\ell^m(y))\,\Falpha(\diff s,\diff y).
\end{aligned}
\end{equation}
Let the operator $e^{\dt\Delta}$ be defined such that for all $v\in\mathcal{C}_0^0(\overline{Q})$ one has
\[
e^{\dt\Delta}v(x)=\int_Q G_d(\dt,x,y)v(y)\diff y,\quad \forall~x\in \overline{Q}.
\]
In addition, define Gaussian random variables $\bigl(\delta_\ell F_\alpha\bigr)_{0\le\ell\le m-1}$ by
\[
\delta_\ell F_\alpha(\varphi)=F_\alpha\bigl(\mathds{1}_{[t_{\ell},t_{\ell+1})}\otimes \varphi\bigr)=\int_{t_{\ell}}^{t_{\ell+1}}\int_Q \varphi(y)\,\Falpha(\diff s,\diff y),\quad \forall~\varphi\in\mathcal{H}_\alpha.
\]
Note that if $\ell_1\neq\ell_2$ then the Gaussian random variables $\delta_{\ell_1}\Falpha$ and $\delta_{\ell_2}\Falpha$ are independent. Moreover, for all $\ell\in\{0,\ldots,m-1\}$, $\delta_\ell\Falpha$ is a centered Gaussian random variable, with covariance structure given by
\[
\E\bigl[\delta_\ell\Falpha(\varphi)\delta_\ell\Falpha(\psi)\bigr]=\tau \langle \varphi,\psi\rangle_{\alpha}=\tau\iint_{Q\times Q}\varphi(x)\psi(y)|x-y|^{-\alpha}\diff x\diff y,\qquad \forall~\varphi,\psi\in\mathcal{H}_\alpha.
\]
Employing the notation introduced above, the stochastic exponential integrator is written as
\begin{equation}\label{sexp}
\mathcal U^m_{\ell+1}=\e^{\dt\Delta}\left(\mathcal U^m_\ell+\dt b(t_\ell,\cdot,\mathcal U_\ell^m)+\sigma(t_\ell,\cdot,\mathcal U_\ell^m)\delta_\ell \Falpha\right),\ \ell = 0,\ldots, m-1.
\end{equation}
The initial value of the scheme is $\mathcal{U}_0^m=u_0$ for any value $\dt=T/m$ of the time-step size.

We are now in position to state the main result of this article.
\begin{theorem}\label{thm:main}
Let Assumptions~\ref{ass:init} and~\ref{ass:coeff} be satisfied. Let $T\in(0,\infty)$. Let $u$ denote the unique mild solution given by~\eqref{mild} to the stochastic partial differential equation~\eqref{prob}. For all $m\in\N$, let $u^m$ denote the solution to the stochastic exponential integrator~\eqref{timeapp} with time-step size $\dt=T/m$.

For all $p\in[1,\infty)$ and $\gamma\in(0,\frac{1}{2}-\frac{\alpha}{4})$, there exists $C_{p,\alpha,\gamma}(T)\in(0,\infty)$ such that for all $m\in\N$ one has
\begin{equation}\label{eq:main}
\sup_{(t,x)\in [0,T]\times \overline{Q}}\E\left[ |u^{m}(t,x)-u(t,x)|^{2p} \right]^{\frac1{2p}} \le	C_{p,{\alpha},\gamma}(T) \dt^{\gamma}.
\end{equation}
\end{theorem}
The proof of Theorem~\ref{thm:main} is postponed to Section~\ref{sec-proofs}.

Note that one obtains the same rate of convergence, arbitrarily close to $\frac12-\frac{\alpha}{4}$, as for the explicit Euler--Maruyama scheme and the semi-implicit Euler--Maruyama scheme obtained in~\cite[Theorem 3.4. part (iii)]{MR2147242}. However, in Theorem~\ref{thm:main} there is no need to impose regularity conditions on the initial value $u_0$ other than continuity.

\section{Numerical experiments}\label{sec-num}
In this section, we provide several numerical experiments in dimensions $1$ and $2$
in order to support and illustrate the theoretical results of this paper.
In addition, we shall compare the behavior of the analyzed stochastic exponential integrator~\eqref{sexp}
with the following classical integrators for SPDEs. Let us recall the notation for all of these integrators:
\begin{itemize}
\item the stochastic exponential scheme~\eqref{sexp} (denoted by \textsc{Sexp} below):
$$
\mathcal U^m_{\ell+1}=\e^{\dt\Delta}\left(\mathcal U^m_\ell+\dt b(t_\ell,\cdot,\mathcal U_\ell^m)+\sigma(t_\ell,\cdot,\mathcal U_\ell^m)\delta_\ell \Falpha\right).
$$
\item the Euler--Maruyama scheme (denoted \textsc{EM}) from \cite{MR2147242}:
$$
\mathcal U^m_{\ell+1}=\mathcal U^m_{\ell}+\dt\Delta\mathcal U^m_{\ell}+\dt b(t_\ell,\cdot,\mathcal U_\ell^m)+\sigma(t_\ell,\cdot,\mathcal U_\ell^m)\delta_\ell\Falpha.
$$
\item the semi-implicit Euler--Maruyama scheme (denoted \textsc{sEM}) from \cite{MR2147242}:
$$
\mathcal U^m_{\ell+1}=\mathcal U^m_{\ell}+\dt\Delta\mathcal U^m_{\ell+1}+ \dt b(t_\ell,\cdot,\mathcal U_\ell^m)+\sigma(t_\ell,\cdot,\mathcal U_\ell^m)\delta_\ell\Falpha.
$$
\end{itemize}
Let us recall that the noise in the parabolic SPDE~\eqref{prob} is given by a Riesz potential $f(r)=r^{-\alpha}$
with parameter $\alpha\in(0,2\wedge d)$.

We start this section by describing the finite difference discretization of the parabolic SPDE~\eqref{prob}
from \cite{MR2147242} in Subsection~\ref{sec:spaceDisc}. We then numerically illustrate the profile in Subsection~\ref{sec:prof1d},
the strong rate of convergence in Subsection~\ref{sec:strong1d}, and the computational costs
of the integrators (see Subsection~\ref{sec:cpu1d}) applied to the SPDE~\eqref{prob} in spatial dimension $d=1$.
Finally, strong rates of convergence and computational costs of
the integrators are presented for our parabolic SPDE in spatial dimension $d=2$ in Subsection~\ref{sec:strong2d},
resp. Subsection~\ref{sec:cpu2d}.

\subsection{Spatial discretization of the parabolic SPDE}\label{sec:spaceDisc}
In this subsection, we briefly recall the finite difference scheme for the parabolic SPDE~\eqref{prob} studied in \cite{MR2147242}.

Fix a positive integer $n$ and consider a uniform grid
$\frac{\uk}{n}=(\frac{k_1}{n},\ldots,\frac{k_d}{n})$ of $\overline{Q}$, where
$k_j\in\{0,\ldots,n\}$ for $1\leq j\leq d$. For any non-negative
integer $i$, set $x_i:=\frac{i}{n}$. In dimension $d=1$, a centered
finite difference approximation of the Laplacian is provided by the $(n-1)\times (n-1)$ matrix $n^2D_n$.
The matrix $D_n$ is given by
$$
D_n=
\begin{pmatrix}
-2 & 1 & 0 & \ldots & 0 \\
1 & -2 & 1 & \ddots & \vdots \\
0 & \ddots & \ddots & \ddots & 0 \\
\vdots & \ddots & 1 & -2 & 1 \\
0 & \vdots & 0 & 1 & -2
\end{pmatrix}.
$$
For $d\geq2$, the spatial approximation of the
Laplacian in the parabolic SPDE~\eqref{prob} is recursively given by
$$
D_n^{(d)}=\diag(D_n^{(d-1)})+\begin{pmatrix}-2Id_{(n-1)^{d-1}} & Id_{(n-1)^{d-1}} & 0 & \ldots & 0 \\
Id_{(n-1)^{d-1}} & -2Id_{(n-1)^{d-1}} & Id_{(n-1)^{d-1}} & \ddots & \vdots \\
0 & \ddots & \ddots & \ddots & 0 \\
\vdots & \ddots & Id_{(n-1)^{d-1}} & -2Id_{(n-1)^{d-1}} & Id_{(n-1)^{d-1}} \\
0 & \vdots & 0 & Id_{(n-1)^{d-1}} & -2Id_{(n-1)^{d-1}}\end{pmatrix},
$$
starting with $D_n^{(1)}=D_n$, where
$Id_{(n-1)^{d-1}}$ is the $(n-1)^{d-1}\times (n-1)^{d-1}$ identity
matrix and $\diag(D_n^{(d-1)})$ denotes the
$(n-1)^{d}\times(n-1)^{d}$ matrix with $d$ diagonal blocks equal to
the $(n-1)^{d-1}\times(n-1)^{d-1}$ matrix $D_n^{(d-1)}$. We denote
by $u^n(t)$ the $(n-1)^d$-dimensional vector approximating the
solution at time $t$ to the SPDE~\eqref{prob} on the grid
points. For any $k_1,\dots,k_d\in \{1,\ldots,n-1\}$, we define the
non-negative integer $\mathbf
k=(k_d-1)(n-1)^{d-1}+\ldots+(k_2-1)(n-1)+k_1$ and set
$u^n(t,{\mathbf x}_{\mathbf k}):=(u^n(t)_{\mathbf k})_{\mathbf
k=1}^{(n-1)^d}$, with ${\mathbf x}_{\mathbf
k}:=(x_{k_1},\ldots,x_{k_d})$, which corresponds to a spatial
approximation of $u(t,{\mathbf x}_{\mathbf k})$. Then, the vector
$u^n(t)$ satisfies the stochastic differential equation
\begin{align}\label{FDsde}
\diff u^n(t)=n^2D_n^{(d)}u^n(t)\diff t+b(u^n(t))\diff
t+\sigma(u^n(t))\diff \Falpha^n(t),
\end{align}
with initial values $u^n(0)_{\mathbf k}=u_0({\mathbf x}_{\mathbf
k})$. Here, the following classical convention is used: for any
function $h:\mathbb{R}\rightarrow \mathbb{R}$ and vector $v\in
\mathbb{R}^r$, we write $h(v):=(h(v_1),\dots,h(v_r))$ and interpret the last product in \eqref{FDsde} in the elementwise sense.
The vector $\Falpha^n(t)$ is given by
\[
\Falpha^n(t)_{\bold k}:= n^d \int_0^t \int_{\msquare_{{\bold x}_{\bold
k}}}\Falpha(\diff s,\diff y),
\]
where $\msquare_{{\bold x}_{\bold k}}=\{x=(x_1,\ldots,x_d)\,;\: x_{k_j}\leq x \leq x_{k_j+1},\,\forall~1\leq j\leq d\}$.

Note that, in order to apply a time integrator to the SDE~\eqref{FDsde}, one needs to know the covariance matrix of the discretized noise.
An explicit formula can be obtain in the $1$-dimensional case, see Appendix~\ref{app:1Dnoise} for details. To the best of our knowledge, this is
not the case in the $2$-dimensional case, where we employ quadrature formulas,
see Appendix~\ref{app:2Dnoise} for details.

\subsection{Time evolution and profile of the solution in dimension $\mathbf1$}\label{sec:prof1d}
We consider the semilinear stochastic heat equation \eqref{prob} on the interval $[0,1]$
with homogeneous Dirichlet boundary conditions and in the time interval $[0,0.5]$.
For this numerical experiment, we choose $b(u)=\sigma(u)=1+0.5\cos(u)$ and the initial value $u_0(x)=\sin(\pi x)$}.
We illustrate the behavior of the solution for three values of the parameter $\alpha$ of the Riesz noise: $\alpha=0.2, 0.7$ and $\alpha=1$
(in order to denote a space-time white noise (which is not a Riesz noise) as in \cite{MR2147242}).
A proof of convergence of the stochastic exponential scheme for space-time white noise is given in \cite{MR4050540} for instance.
The finite difference method is employed with $n=2^{10}$ number of grid points and the stochastic exponential integrator
with $\dt=2^{-16}$. We display the time evolution as well as a profile of the numerical solution at $T_{end}=0.5$
in Figure~\ref{fig:evo}. In this figure, one can observe that the spatial and temporal regularity depends on $\alpha$,
see Proposition~\ref{prop:reg}.

\begin{figure}[h]
\begin{subfigure}{.5\textwidth}
 \centering
 \includegraphics[width=.5\linewidth]{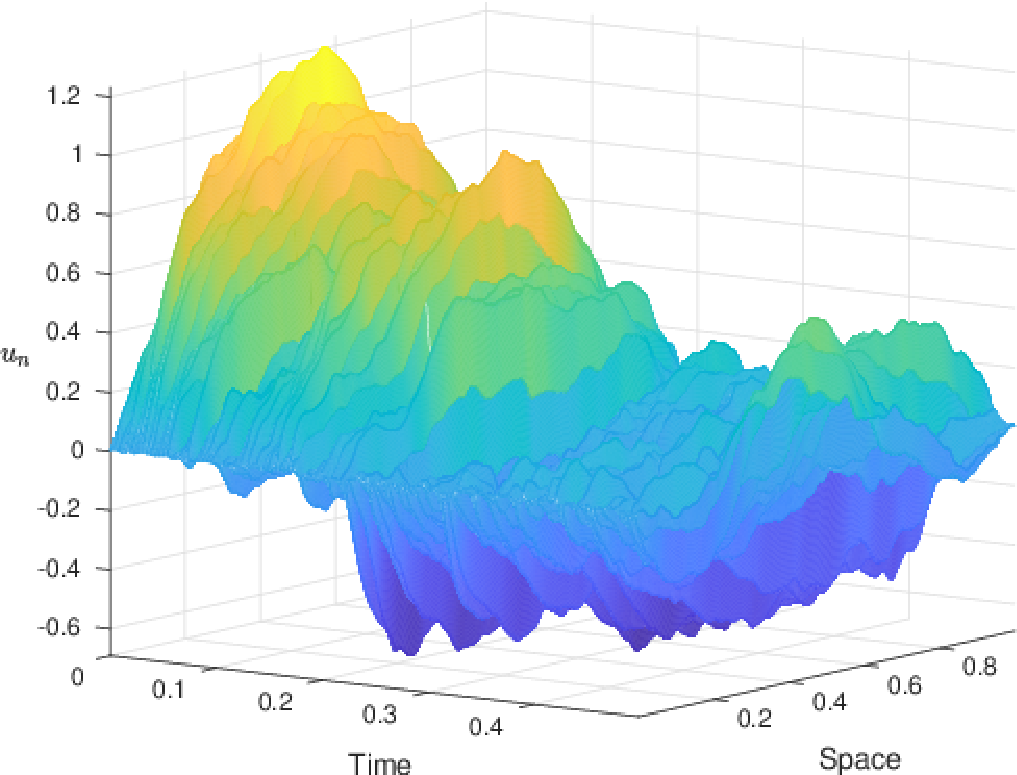}
 \caption*{$\alpha=0.2$}
\end{subfigure}%
\begin{subfigure}{.5\textwidth}
 \centering
 \includegraphics[width=.5\linewidth]{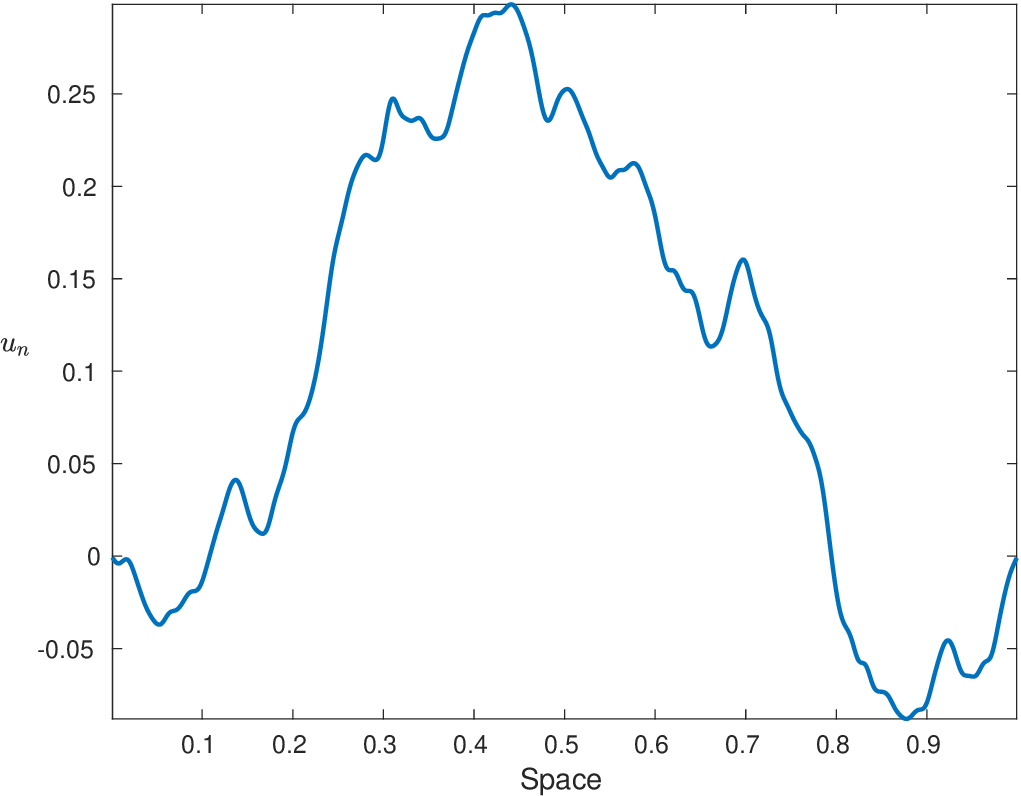}
 \caption*{$\alpha=0.2$}
\end{subfigure}
\begin{subfigure}{.5\textwidth}
 \centering
 \includegraphics[width=.5\linewidth]{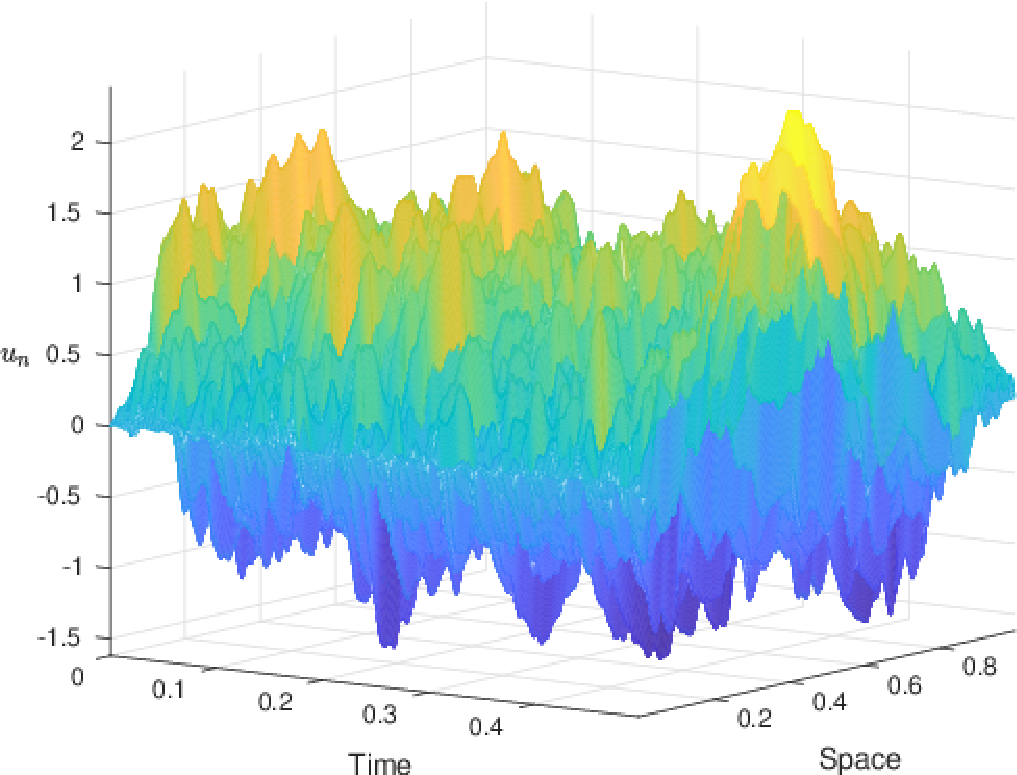}
 \caption*{$\alpha=0.7$}
\end{subfigure}%
\begin{subfigure}{.5\textwidth}
 \centering
 \includegraphics[width=.5\linewidth]{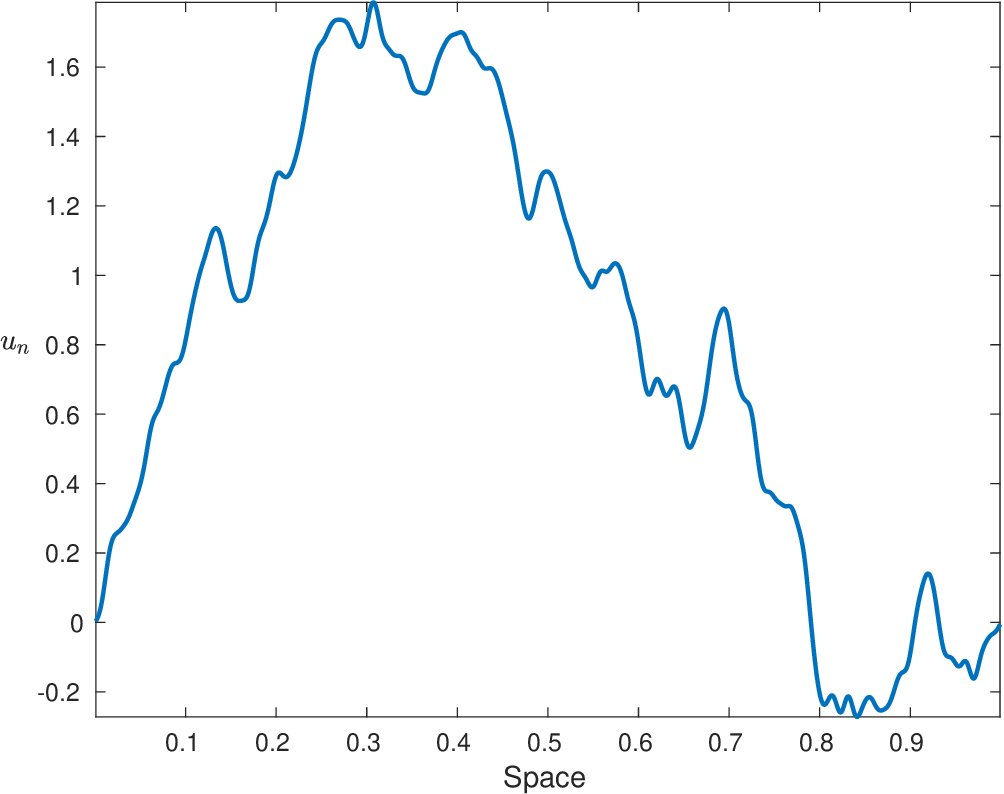}
 \caption*{$\alpha=0.7$}
\end{subfigure}
\begin{subfigure}{.5\textwidth}
 \centering
 \includegraphics[width=.5\linewidth]{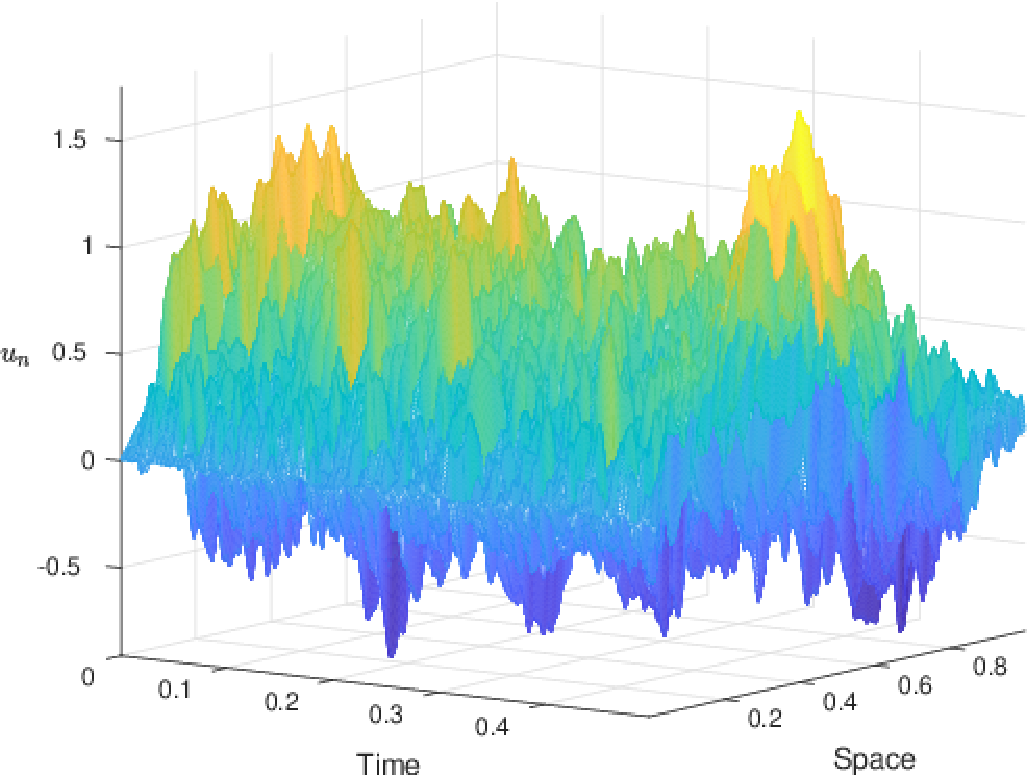}
 \caption*{$\alpha=1$}
\end{subfigure}%
\begin{subfigure}{.5\textwidth}
\centering
 \includegraphics[width=.5\linewidth]{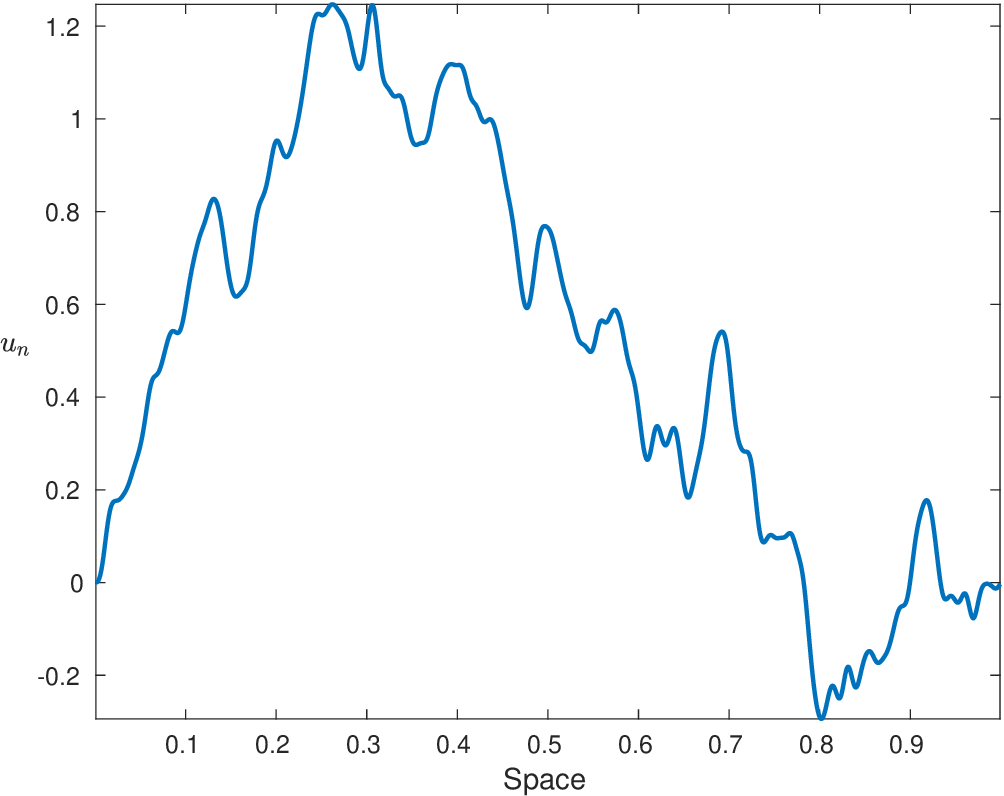}
 \caption*{$\alpha=1$}
\end{subfigure}
\caption{Time evolution (left) and profile at $T_{end}=0..5$ (right) for different values of the parameter $\alpha$.}
\label{fig:evo}
\end{figure}

\subsection{Strong convergence in dimension $\mathbf1$}\label{sec:strong1d}
In this subsection we illustrate the strong convergence of the stochastic exponential integrator~\eqref{sexp}
as stated in Theorem~\ref{thm:main}.

We consider the semilinear SPDE~\eqref{prob} for the time interval $[0,0.5]$.
We consider $b(u)=\sigma(u)=1+0.5\cos(u)$ and the initial value $u_0(x)=\sin(\pi x)$.
The parameter for the Riesz noise is set to be $\alpha=0.7$. In addition, we use $100$ samples
to approximate the expectations. We have verified that this is enough for the Monte Carlo errors to be negligible.
We apply the finite difference method with $n=2^{9}$ and the stochastic exponential integrator
with $\dt_{ref}=2^{-20}$ to produce the reference solution $u_{ref}$.
The strong errors
$$
\sup_{(t_j,x_k)\in[0,0.5]\times[0,1]}\E\left[\abs{u^{j,k}-u_{ref}(t_j,x_k)}^2\right]
$$
of the considered integrators are presented in Figure~\ref{fig:strong} (left). Here, $u^{j,k}$ denotes a numerical approximation
of $u(t_j,x_k)$. An order of convergence $1-\alpha/2=0.65$ for the proposed time integrator~\eqref{sexp} is observed.
This is in agreement with the results of Theorem~\ref{thm:main}.

We now illustrate the dependence of the order of convergence of the stochastic exponential integrator with respect to the parameter $\alpha$
and consider a smoother noise with the parameter $\alpha=0.2$
(the other parameters for this simulation are as above). The results are presented in Figure~\ref{fig:strong} (right).
Strong order of convergence $1-\alpha/2=0.9$ is observed for the stochastic exponential integrator~\eqref{sexp}.
This is in agreement with Theorem~\ref{thm:main}. Figure~\ref{fig:strong} also illustrates the rates of convergence
of the semi-implicit Euler--Maruyama scheme, as proved in the work \cite{MR2147242}, as well as the (well-known)
fact that the classical Euler--Maruyama scheme has a severe step size restriction when applied to (S)PDEs.

\begin{figure}[h]
\begin{subfigure}{.5\textwidth}
\centering
\includegraphics[width=1.\linewidth]{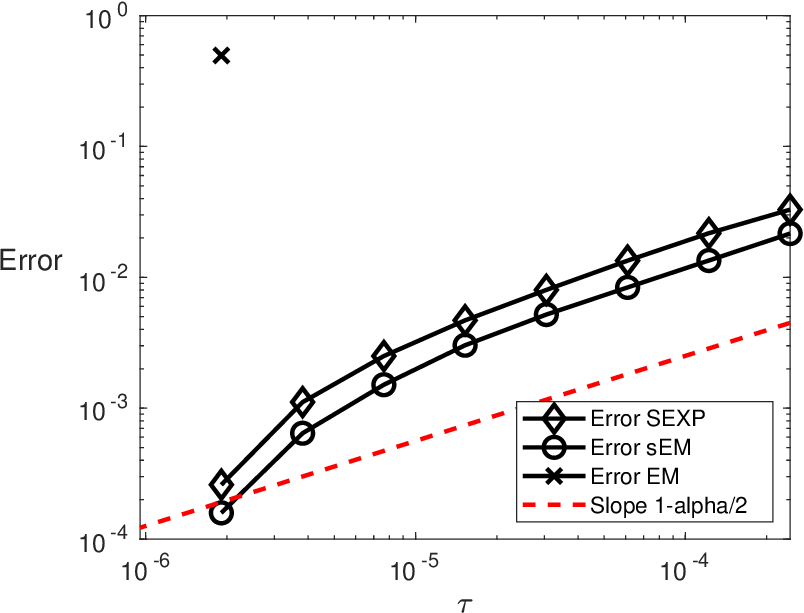}
\caption{$\alpha=0.7$}
\end{subfigure}%
\begin{subfigure}{.5\textwidth}
\centering
\includegraphics[width=1.\linewidth]{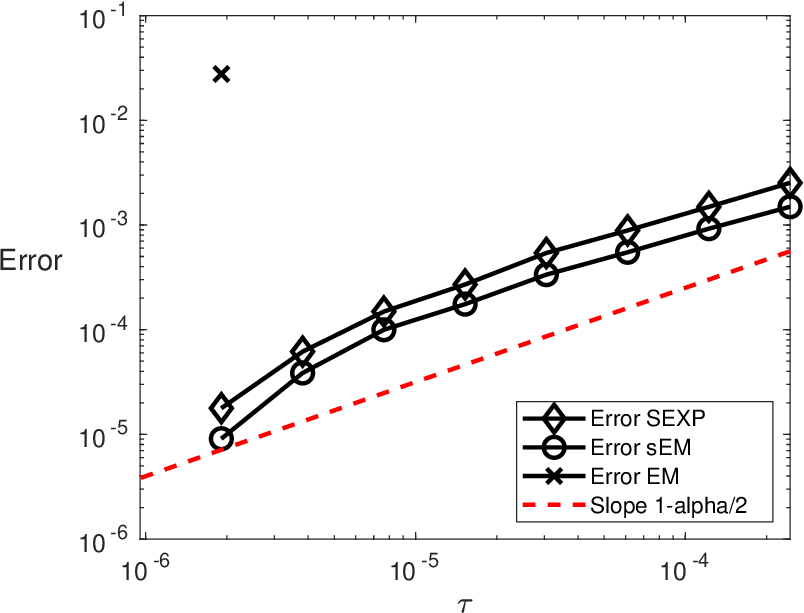}
\caption{$\alpha=0.2$}
\end{subfigure}
\caption{Strong errors for the SPDE~\eqref{prob} with Riesz potential $f(r)=r^{-\alpha}$.}
\label{fig:strong}
\end{figure}

We conclude this section by further illustrating the dependence of the strong order of convergence of the proposed time integrator with respect
to the parameter of the noise $\alpha$. We perform another numerical experiment with the following parameters: $T=0.5$, $u_0(x)=\sin(\pi x)$, $b(u)=2+\sin(u)$, $\sigma(u)=u+2$, $n=2^8$, $\dt_{ref}=2^{-20}$, $\dt=2^{-9},\ldots,2^{-16}$, $150$ Monte Carlo samples (this is enough for the Monte Carlo errors to be negligible), and $\alpha=0.2, 0.4, 0.6, 0.8$. The results are presented in Figure~\ref{fig:strong2} and
further illustrate Theorem~\ref{thm:main}.

\begin{figure}[h]
\includegraphics[width=.5\linewidth]{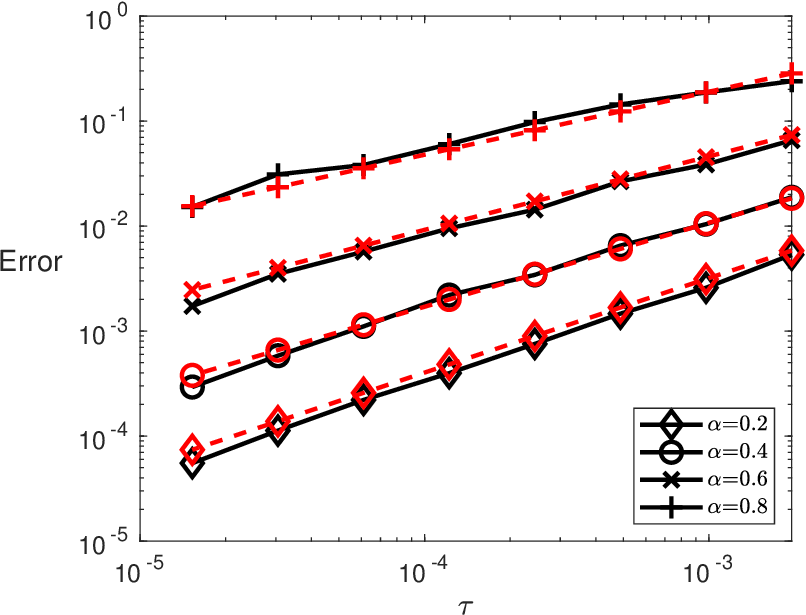}
\caption{Strong errors for the SPDE~\eqref{prob} with Riesz potential $f(r)=r^{-\alpha}$ for $\alpha=0.2,0.4,0.6,0.8$.}
\label{fig:strong2}
\end{figure}

\subsection{Computational times in dimension $\mathbf1$}\label{sec:cpu1d}
In this subsection we compare the computational cost of the above time integrators.
To this end, we consider the SPDE~\eqref{prob}. We run $50$ samples for each numerical scheme.
For each numerical integrator and each sample,
we use several time-step sizes ($2^{-\ell} \dt_{ref}$ for $\ell=1, 2, \ldots, 10$) and compare the strong error at the final time $T=0.75$ with a reference solution computed with the same sample of the noise and by the same numerical scheme for the time-step size $\dt_{ref}=2^{-20}$.
Figure~\ref{fig:compcos} displays the total computational time for all the samples,
for each numerical scheme and each time-step size, as a function of the averaged final error.
In this figure, one observes better performance for the stochastic exponential integrator~\eqref{sexp}
than for the two classical time integrators from \cite{MR2147242}. In addition, one can observe a time-step restriction
for the explicit Euler--Maruyama scheme (top right of the figure).

\begin{figure}[h]
\centering
\includegraphics*[height=5cm,keepaspectratio]{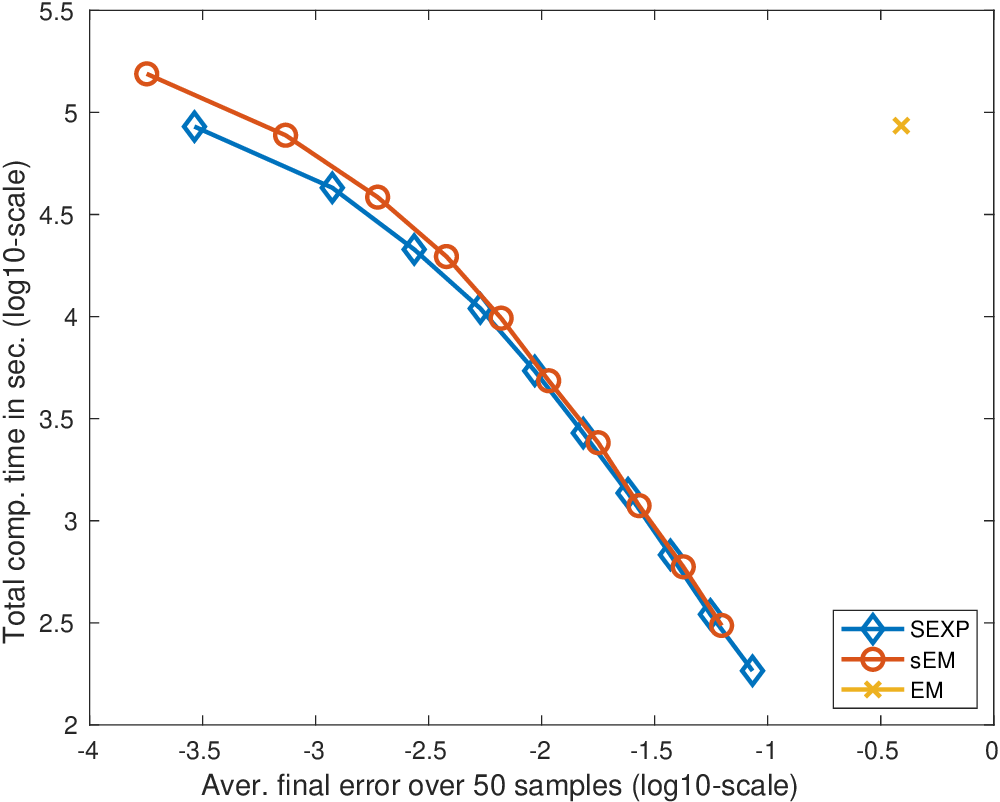}
\caption{Computational time as a function of the averaged final error for the three numerical methods.}
\label{fig:compcos}
\end{figure}

\subsection{Strong convergence in dimension $\mathbf2$}\label{sec:strong2d}
In this subsection, we confirm the theoretical result in Theorem~\ref{thm:main} in dimension $d=2$. We consider the semilinear SPDE~\eqref{prob} on the spatial domain $[0,1]^{2}$ with homogeneous Dirichlet boundary conditions and for the time interval $[0,1]$. We consider $b(u)=\sigma(u)=1+\cos(u)$ and the initial value $u_{0}(x,y) = \sin (2 \pi x) \sin (2 \pi y)$ for $x,y \in [0,1]$.

We compute the strong errors
$$
\left( \sup_{(t_k,x_i,x_{j})\in[0,1]\times[0,1]\times[0,1]}\E\left[\abs{u^{k,i,j}-u_{ref}(t_k,x_i,x_{j})}^2\right] \right)^{1/2},
$$
where the reference solution $u_{ref}$ is computed using the stochastic exponential integrator with temporal discretization size
$\tau_{ref} = 2^{-13}$. We use $100$ samples to approximate the expectation in the strong errors and
we have verified that this is enough for the Monte Carlo
errors to be negligible.

In the first numerical experiment, we fix the space discretization parameter to $n=2^{6}=64$ and take the following values of the noise correlation parameter $\alpha=0.2, 0.4, 0.6, 0.8, 1, 1.2, 1.4, 1.6, 1.8$. As is seen in Figures~\ref{fig:N64_alpha1}~and~\ref{fig:N64_alpha2}, the decay of the strong errors follow the reference lines with slopes $1/2-\alpha/4$. This confirms the results derived in Theorem~\ref{thm:main}.

\begin{figure}[h]
\centering
\includegraphics*[height=5cm,keepaspectratio]{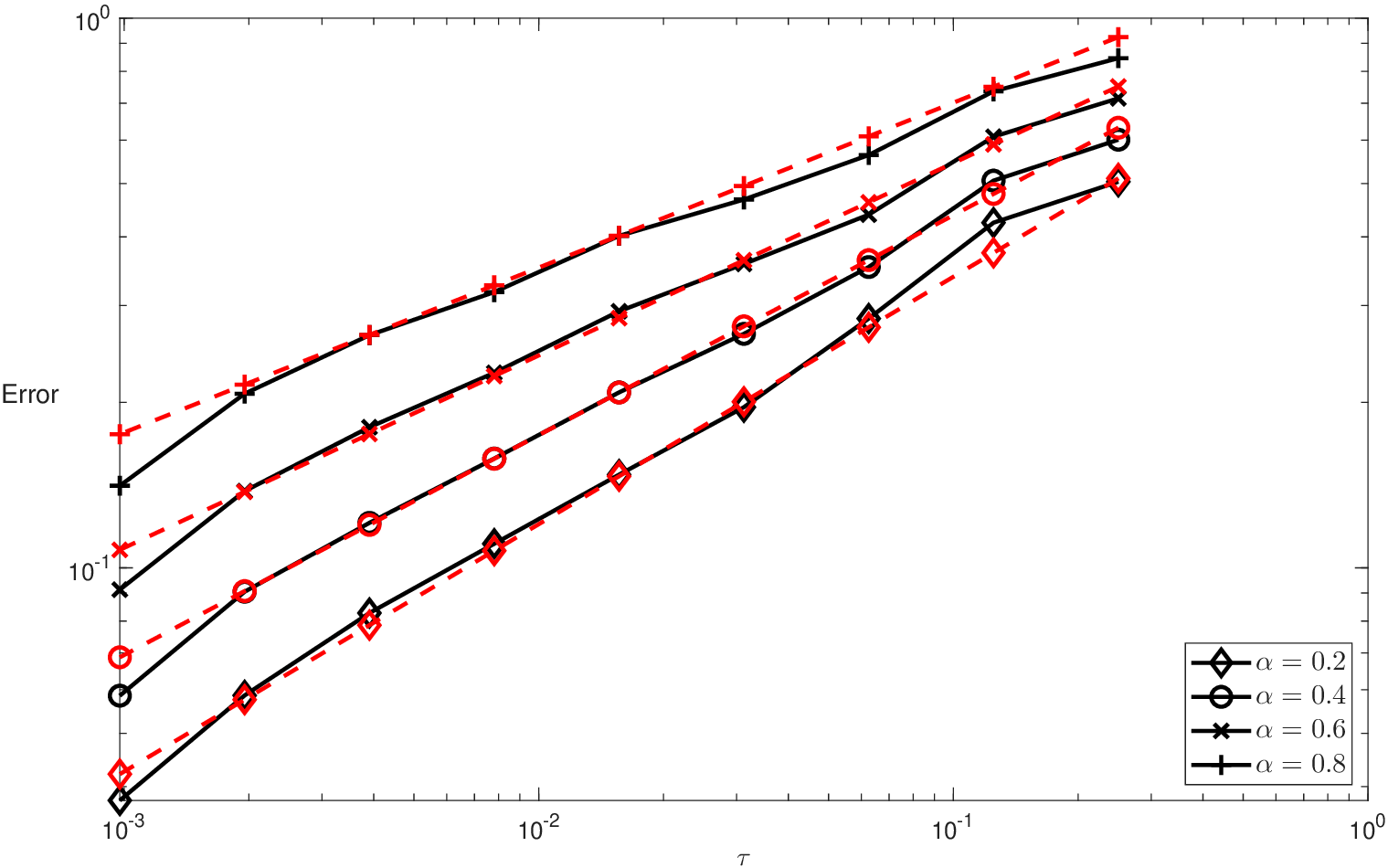}
\caption{Root mean square convergence of the stochastic exponential Euler scheme~\eqref{sexp} for $\alpha = 0.2, 0.4, 0.6, 0.8$ and for $n = 64$ space grid point in each direction.}
\label{fig:N64_alpha1}
\end{figure}

\begin{figure}[h]
\centering
\includegraphics*[height=5cm,keepaspectratio]{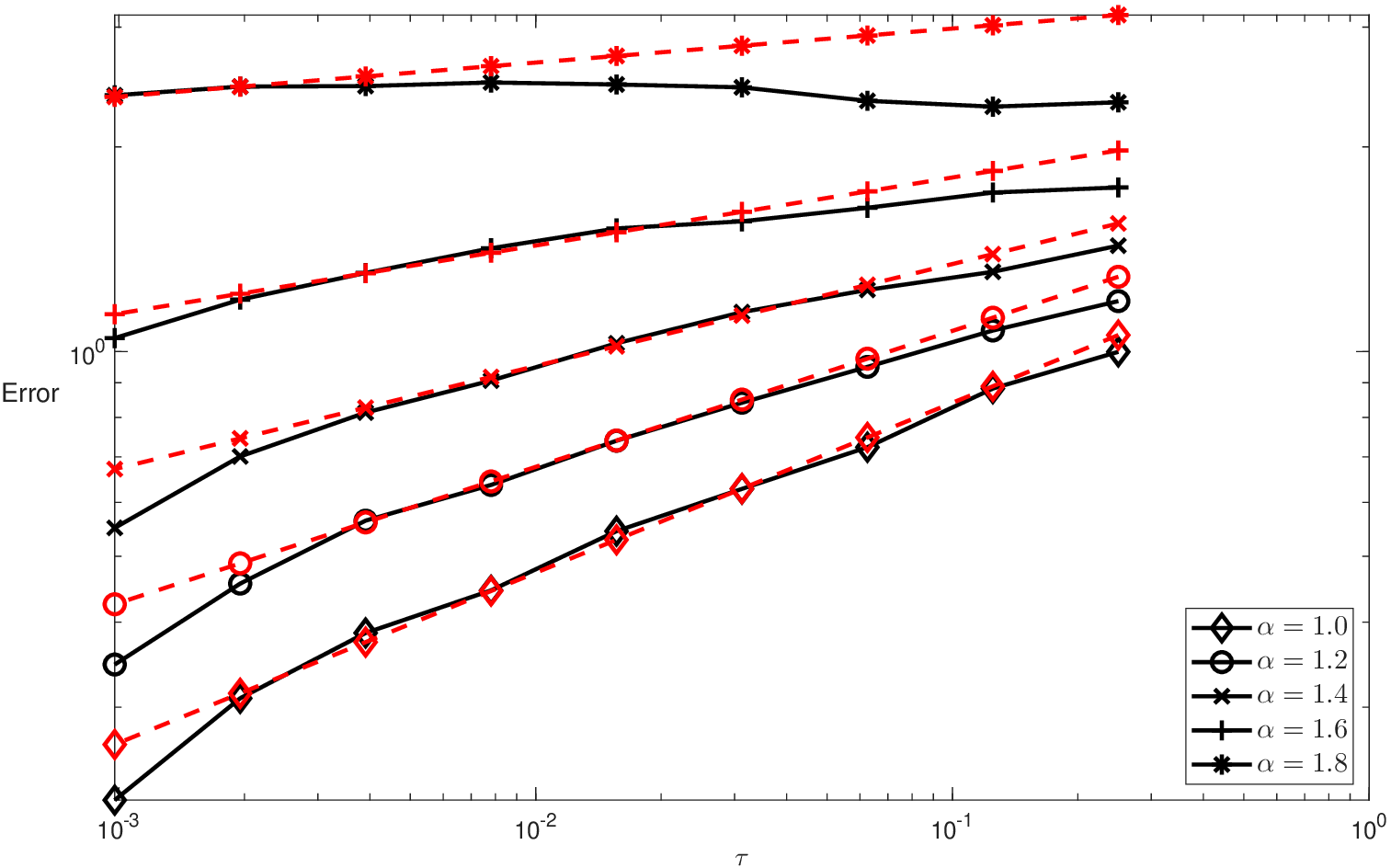}
\caption{Root mean square convergence of the stochastic exponential Euler scheme~\eqref{sexp} for $\alpha = 1, 1.2, 1.4, 1.6, 1.8$ and for $n = 64$ space grid points.}
\label{fig:N64_alpha2}
\end{figure}

In the second numerical experiment, presented in Figure~\ref{fig:alpha8_Nb}, we fix the noise parameter to $\alpha = 0.8$ and use different values of the space discretization parameter $n = 16, 32, 64$ to confirm that the error in Theorem~\ref{thm:main} is independent of $n$. In addition, we show that the convergence rate agrees with the derived theoretical rate $1/2-\alpha/4$.

\begin{figure}[h]
\centering
\includegraphics*[height=5cm,keepaspectratio]{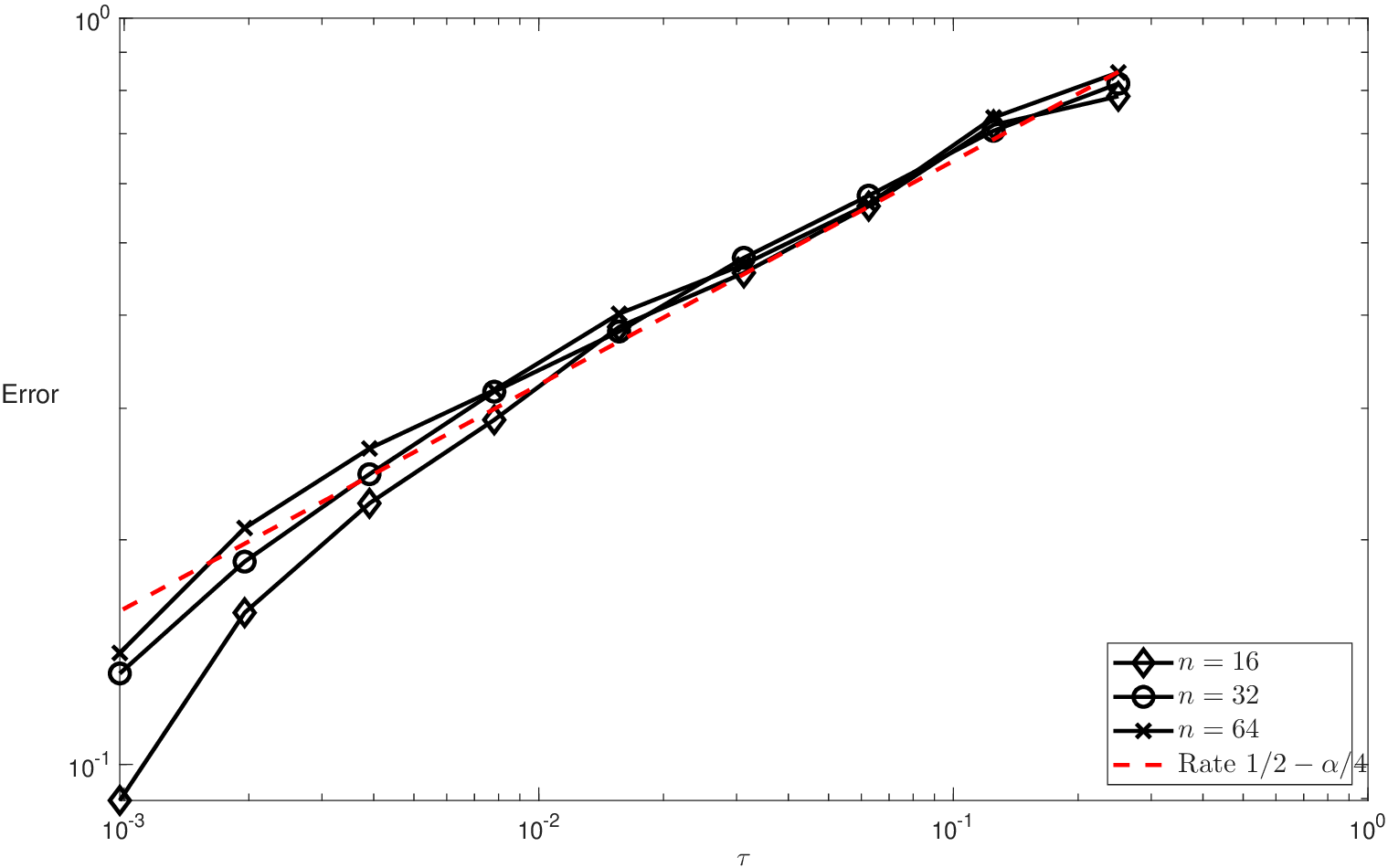}
\caption{Root mean square convergence of the stochastic exponential Euler scheme~\eqref{sexp} for $\alpha = 0.8$ and for $n = 16,32,64$ space grid points.}
\label{fig:alpha8_Nb}
\end{figure}
\
\subsection{Computational times in dimension $\mathbf2$}\label{sec:cpu2d}
In this subsection we compare the computational costs of the stochastic exponential Euler scheme and of the semi-implicit Euler--Maruyama scheme
in dimension $d=2$. We consider the SPDE~\eqref{prob} with the same parameters as in Subsection~\ref{sec:strong2d}. We apply the finite difference discretization with mesh size $n=2^5=32$ and these two integrators with time-step sizes $2^{-l}$ for $l=2,\ldots,2^{18}$. The strong error is
computed using $\tau_{ref}=2^{-19}$ for the reference solution and using $100$ Monte Carlo samples to approximate the expectations.

Figure~\ref{fig:compcos2d} displays the total computational time for the \textsc{SEXP} and \textsc{sEM} schemes for all samples;
meaning that, only the parts of the implementations of the \textsc{SEXP} and \textsc{sEM} schemes that are different are included in the elapsed times. In this figure, one can observe better performance for the \textsc{SEXP} scheme compared to the \textsc{sEM} scheme
in the regime of small errors.

\begin{figure}[h]
\centering
\includegraphics*[height=5cm,keepaspectratio]{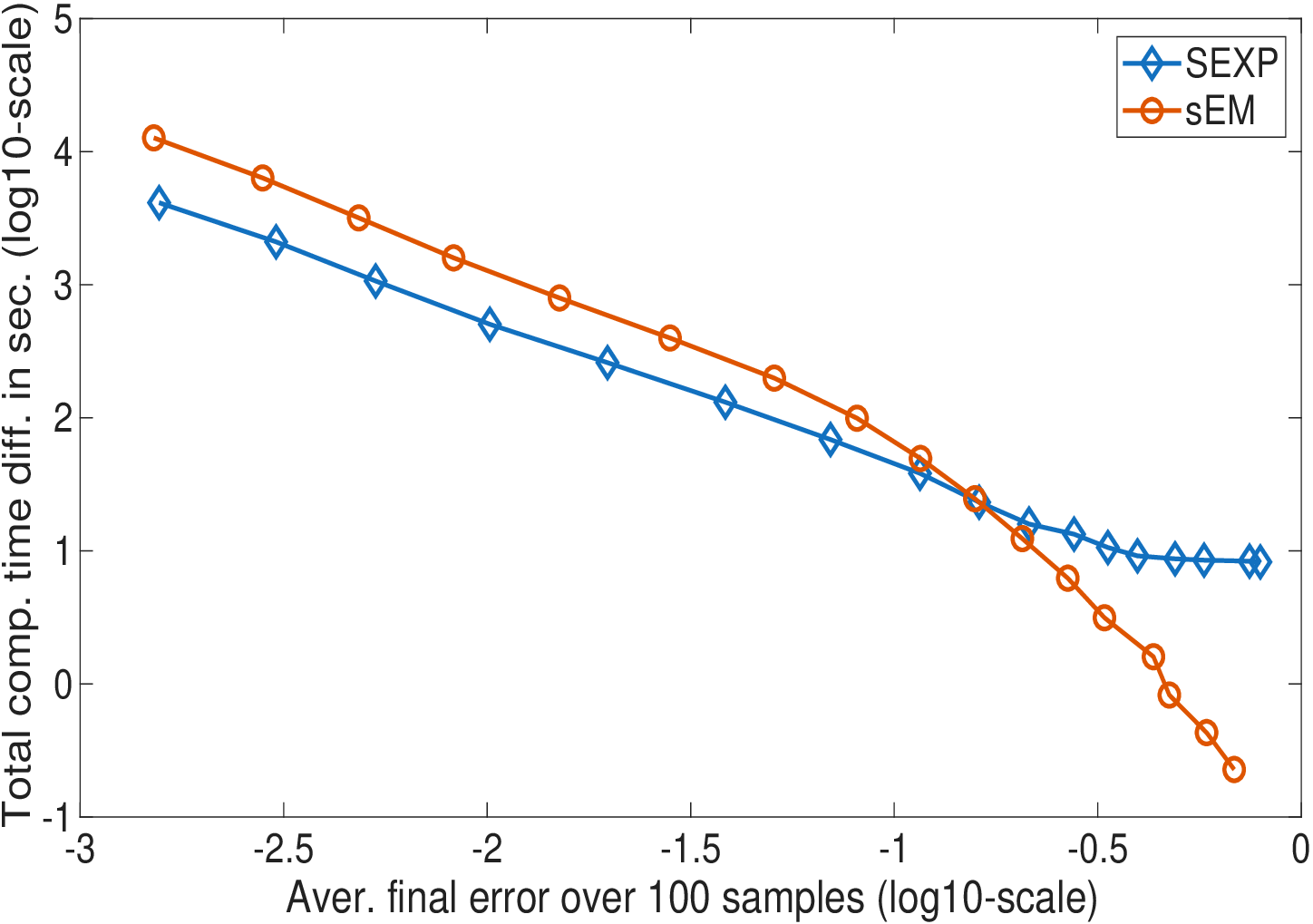}
\caption{Computational time as a function of the averaged final error for the two time integrators \textsc{sEM} and \textsc{SEXP} in dimension $2$.}
\label{fig:compcos2d}
\end{figure}

\section{Proof of the main result}\label{sec-proofs}
This section gives the proof of Theorem~\ref{thm:main}.

\begin{proof}
Recall that the random field $u$ is the mild solution given by~\eqref{mild} to the stochastic partial differential equation~\eqref{prob}, and that, for all $m\in\N$, the random field $u^m$ is the numerical solution given by the stochastic exponential integrator~\eqref{timeapp} with time-step size $\tau=T/m$.

For all $p\in[1,\infty)$, define
\begin{align*}
\mathcal{E}_{2p}^m(t,x)&=\vvvert u(t,x)-u^m(t,x)\vvvert_{2p}=\bigl(\E[|u(t,x)-u^m(t,x)|^{2p}]\bigr)^{\frac{1}{2p}},\quad \forall~(t,x)\in[0,T]\times \overline{Q},\\
\overline{\mathcal{E}_{2p}^m}(t)&=\underset{x\in \overline{Q}}\sup~\mathcal{E}_{2p}^m(t,x),\quad \forall~t\in[0,T].
\end{align*}

We consider the following decomposition of the error $u(t,x)-u^m(t,x)$ for all $t\in[0,T]$ and $x\in \overline{Q}$: one has
\begin{align*}
u(t,x)-u^m(t,x)&=A_1^m(t,x)+A_2^m(t,x)+A_3^m(t,x)+A_4^m(t,x)\\
&+B_1^m(t,x)+B_2^m(t,x)+B_3^m(t,x)+B_4^m(t,x),
\end{align*}
where the error terms $A_j^m(t,x)$, for $j\in\{1,\ldots,4\}$, are defined as
\begin{align*}
A_1^m(t,x)&=\int_0^t \int_Q \left[G_d(t-s,x,y)-G_d(t-\kappa_m^T(s),x,y)\right]b(s,y,u(s,y)) \diff y \diff s,\\
A_2^m(t,x)&=\int_0^t \int_Q G_d(t-\kappa_m^T(s),x,y) \left[b(s,y,u(s,y))-b(\kappa_m^T(s),y,u(s,y))\right]\diff y \diff s,\\
A_3^m(t,x)&=\int_0^t \int_Q G_d(t-\kappa_m^T(s),x,y) \left[b(\kappa_m^T(s),y,u(s,y))-b(\kappa_m^T(s),y,u(\kappa_m^T(s),y))\right]\diff y \diff s,\\
A_4^m(t,x)&=\int_0^t \int_Q G_d(t-\kappa_m^T(s),x,y) \left[b(\kappa_m^T(s),y,u(\kappa_m^T(s),y))-b(\kappa_m^T(s),y,u^m(\kappa_m^T(s),y))\right]\diff y \diff s,
\end{align*}
and the error terms $B_j^m(t,x)$, for $j\in\{1,\ldots,4\}$, are defined as
\begin{align*}
B_1^m(t,x)&=\int_0^t \int_Q \left[G_d(t-s,x,y)-G_d(t-\kappa_m^T(s),x,y)\right]\sigma(s,y,u(s,y)) \Falpha(\diff s,\diff y),\\
B_2^m(t,x)&=\int_0^t \int_Q G_d(t-\kappa_m^T(s),x,y) \left[\sigma(s,y,u(s,y))-\sigma(\kappa_m^T(s),y,u(s,y))\right]\Falpha(\diff s,\diff y),\\
B_3^m(t,x)&=\int_0^t \int_Q G_d(t-\kappa_m^T(s),x,y) \left[\sigma(\kappa_m^T(s),y,u(s,y))-\sigma(\kappa_m^T(s),y,u(\kappa_m^T(s),y))\right]\Falpha(\diff s,\diff y),\\
B_4^m(t,x)&=\int_0^t \int_Q G_d(t-\kappa_m^T(s),x,y) \left[\sigma(\kappa_m^T(s),y,u(\kappa_m^T(s),y))-\sigma(\kappa_m^T(s),y,u^m(\kappa_m^T(s),y))\right]\Falpha(\diff s,\diff y).
\end{align*}
For all $j\in\{1,\ldots,4\}$ and all $t\in[0,T]$, define
\begin{align*}
&\mathcal{A}_{j,2p}^m(t,x)=\vvvert A_j^m(t,x) \vvvert_{2p},\quad \mathcal{B}_{j,2p}^m(t,x)=\vvvert B_j^m(t,x) \vvvert_{2p},\qquad \forall~x\in Q,\\
&\overline{\mathcal{A}_{j,2p}^m}(t)=\underset{x\in \overline{Q}}\sup~\mathcal{A}_{j,2p}^m(t,x),\quad \overline{\mathcal{B}_{j,2p}}^m(t)=\underset{x\in \overline{Q}}\sup~\mathcal{B}_{j,2p}^m(t,x).
\end{align*}
Applying the Minkowski inequality, one has
\begin{align*}
\mathcal{E}_{2p}^{m}(t,x)&\le \sum_{j=1}^{4}\mathcal{A}_{j,2p}^m(t,x)+\sum_{j=1}^{4}\mathcal{B}_{j,2p}^m(t,x),\qquad \forall~(t,x)\in[0,T]\times Q,\\
\overline{\mathcal{E}}_{2p}^m(t)&\le \sum_{j=1}^{4}\overline{\mathcal{A}_{j,2p}^m}(t)+\sum_{j=1}^{4}\overline{\mathcal{B}_{j,2p}^m}(t),\qquad \forall~t\in[0,T].
\end{align*}
The proof proceeds by proving upper bounds for each error term $\overline{\mathcal{A}_{j,2p}^m}(t)$ and $\overline{\mathcal{B}_{j,2p}^m}(t)$ with $j\in\{1,\ldots,4\}$.

In the sequel, the value of the parameter $\gamma\in(\frac12-\frac{\alpha}{4})$ is fixed.
Note that for all $s\in[0,T]$ one has $\kappa_m^T(s)\le s$ and $|s-\kappa_m^T(s)|\le \dt$.

\medskip

$\bullet$ Treatment of the error term $\overline{\mathcal{A}_{1,2p}^m}(t)$.

Applying the Minkowski inequality, for all $(t,x)\in[0,T]\times \overline{Q}$, one has
\begin{align*}
\vvvert A_1^m(t,x) \vvvert_{2p} &\le \int_0^t \int_Q \big|G_d(t-s,x,y)-G_d(t-\kappa_m^T(s),x,y)\big| \vvvert b(s,y,u(s,y)) \vvvert_{2p} \diff y \diff s\\
&\le \underset{s\in[0,T]}\sup~\underset{y\in \overline{Q}}\sup~\vvvert b(s,y,u(s,y)) \vvvert_{2p}\int_0^t \int_Q \big|G_d(t-s,x,y)-G_d(t-\kappa_m^T(s),x,y)\big| \diff y \diff s.
\end{align*}
On the one hand, owing to the condition~\eqref{LG} from Assumption~\ref{ass:coeff}, the mapping $b$ has at most linear growth. Applying the moment bounds~\eqref{eq:momo-u} from Proposition~\ref{prop:uExistence} gives the upper bounds
\[
\underset{s\in[0,T]}\sup~\underset{y\in \overline{Q}}\sup~\vvvert b(s,y,u(s,y)) \vvvert_{2p}\le C\bigl(1+\underset{s\in[0,T]}\sup~\underset{y\in \overline{Q}}\sup~\vvvert u(s,y)\vvvert_{2p}\bigr)\le C_p(T)(1+\norm{u_0}_\infty).
\]
On the other hand, applying the inequality~\eqref{eq:GdEstimates-interpoled} on the heat kernel (see Subsection~\ref{sec:kernel}), one obtains the upper bounds
\begin{align*}
\int_0^t \int_Q \big|G_d(t-s,x,y)-G_d(t-\kappa_m^T(s),x,y)\big| \diff y \diff s&\le C\int_0^t \frac{|s-\kappa_m^T(s)|^\gamma}{(t-s)^\gamma}\int_Q (t-s)^{-\frac{d}{2}}e^{-c\frac{|y-x|^2}{2(t-s)}} \diff y \diff s \\
&+C\int_0^t \frac{|s-\kappa_m^T(s)|^\gamma}{(t-s)^\gamma}\int_Q (t-\kappa_m^T(s))^{-\frac{d}{2}}e^{-c\frac{|y-x|^2}{2(t-\kappa_m^T(s))}} \diff y \diff s\\
&\le C\int_0^t \frac{\dt^\gamma}{(t-s)^\gamma} \diff s\\
&\le C_\gamma(T) \dt^\gamma.
\end{align*}
As a result, for the error term $\overline{\mathcal{A}_{1,2p}^m}(t)$, one obtains the following upper bound: there exists $C_{p,\gamma}(T)\in(0,\infty)$ such that one has
\begin{equation}\label{eq:errorA1}
\underset{t\in[0,T]}\sup~\overline{\mathcal{A}_{1,2p}^m}(t)\le C_{p,\gamma}(T)(1+\norm{u_0}_\infty)\dt^\gamma.
\end{equation}

\medskip

$\bullet$ Treatment of the error term $\overline{\mathcal{A}_{2,2p}^m}(t)$.

Applying the Minkowski inequality, for all $(t,x)\in[0,T]\times \overline{Q}$, one has
\[
\vvvert A_2^m(t,x) \vvvert_{2p} \le \int_0^t \int_Q G_d(t-\kappa_m^T(s),x,y) \vvvert b(s,y,u(s,y))-b(\kappa_m^T(s),y,u(s,y)) \vvvert_{2p} \diff y \diff s.
\]
Owing to the condition~\eqref{L} from Assumption~\ref{ass:coeff} and to the assumption $\gamma\in(0,\frac12-\frac{\alpha}{4})$, the mapping $s\in[0,T]\mapsto b(s,y,u)$ satisfies a H\"older continuity property with exponent $\frac12-\frac{\alpha}{4}$, uniformly with respect to $(y,u)\in\R^2$. Therefore, one has
\begin{align*}
\vvvert A_2^m(t,x) \vvvert_{2p}&\le C \int_0^t \int_Q G_d(t-\kappa_m^T(s),x,y) |s-\kappa_m^T(s)|^{\frac12-\frac{\alpha}{4}} \diff y \diff s\\
&\le C \dt^{\frac12-\frac{\alpha}{4}} \int_0^t \int_Q G_d(t-\kappa_m^T(s),x,y) \diff y \diff s.
\end{align*}
Applying the property~\eqref{eq:Gdint} of the heat kernel $G_d$ (see Subsection~\ref{sec:kernel}), one obtains
\[
\vvvert A_2^m(t,x) \vvvert_{2p} \le C T \dt^{\frac12-\frac{\alpha}{4}}.
\]
Recall that $\gamma\in(0,\frac12-\frac{\alpha}{4})$. As a result, for the error term $\overline{\mathcal{A}_{2,2p}^m}(t)$, one obtains the following upper bound: there exists $C_{p,\gamma}(T)\in(0,\infty)$ such that one has
\begin{equation}\label{eq:errorA2}
\underset{t\in[0,T]}\sup~\overline{\mathcal{A}_{2,2p}^m}(t)\le C_{p,\gamma}(T) \dt^\gamma.
\end{equation}

\medskip

$\bullet$ Treatment of the error term $\overline{\mathcal{A}_{3,2p}^m}(t)$.

Applying the Minkowski inequality, for all $(t,x)\in[0,T]\times \overline{Q}$, one has
\[
\vvvert A_3^m(t,x) \vvvert_{2p} \le \int_0^t \int_Q G_d(t-\kappa_m^T(s),x,y) \vvvert b(\kappa_m^T(s),y,u(s,y))-b(\kappa_m^T(s),y,u(\kappa_m^T(s),y))\vvvert_{2p}\diff y \diff s.
\]
Owing to the condition~\eqref{L} from Assumption~\ref{ass:coeff}, the mapping $u\mapsto b(s,y,u)$ satisfies a global Lipschitz continuity property. Therefore, applying the property~\eqref{eq:Gdint} of the heat kernel $G_d$ (see Subsection~\ref{sec:kernel}), one has
\begin{align*}
\vvvert A_3^m(t,x) \vvvert_{2p} &\le C \int_0^t \int_Q G_d(t-\kappa_m^T(s),x,y) \vvvert u(s,y)-u(\kappa_m^T(s),y)\vvvert_{2p}\diff y \diff s\\
&\le C \int_0^t \underset{y\in \overline{Q}}\sup~\vvvert u(s,y)-u(\kappa_m^T(s),y)\vvvert_{2p} \int_Q G_d(t-\kappa_m^T(s),x,y) \diff y \diff s\\
&\le C\int_0^t \underset{y\in \overline{Q}}\sup~\vvvert u(s,y)-u(\kappa_m^T(s),y)\vvvert_{2p} \diff s.
\end{align*}
To proceed, one needs to apply the regularity property~\eqref{eq:reg} from Proposition~\ref{prop:reg} of the solution $u$. However, since $\kappa_m^T(s)=0$ for all $s\in[0,\dt)$, it is not applicable on the time interval $[0,\min(t,\dt))$. The integral is then decomposed as
\begin{align*}
\int_0^t \underset{y\in \overline{Q}}\sup~\vvvert u(s,y)-u(\kappa_m^T(s),y)\vvvert_{2p} \diff s&=\int_0^{\min(t,\tau)} \underset{y\in \overline{Q}}\sup~\vvvert u(s,y)-u(\kappa_m^T(s),y)\vvvert_{2p} \diff s\\
&+\mathds{1}_{t>\tau}\int_{\tau}^t \underset{y\in \overline{Q}}\sup~\vvvert u(s,y)-u(\kappa_m^T(s),y)\vvvert_{2p} \diff s.
\end{align*}
On the one hand, note that for all $s\in[0,\min(t,\dt))$ one has $u(\kappa_m^T(s),\cdot)=u_0$, thus applying the Minkowski inequality and the moment bounds~\eqref{eq:momo-u} from Proposition~\ref{prop:uExistence}, for the first integral one has
\begin{align*}
\int_0^{\min(t,\tau)} \underset{y\in \overline{Q}}\sup~\vvvert u(s,y)-u(\kappa_m^T(s),y)\vvvert_{2p} \diff s&\le \int_0^{\min(t,\tau)}  \bigl(\underset{s\in [0,T]}\sup~\underset{y\in \overline{Q}}\sup~\vvvert u(s,y)\vvvert_{2p}+\norm{u_0}_\infty\bigr)\diff s\\
&\le C\tau\bigl(1+\norm{u_0}_\infty\bigr).
\end{align*}
On the other hand, applying the regularity property~\eqref{eq:reg} from Proposition~\ref{prop:reg}, for all $t\in(\tau,\infty)$ one has
\[
\int_{\tau}^t \underset{y\in Q}\sup~\vvvert u(s,y)-u(\kappa_m^T(s),y)\vvvert_{2p} \diff s \le C_{p,\gamma}(T)\bigl(1+\norm{u_0}_\infty\bigr)\int_\dt^t \frac{|s-\kappa_m^T(s)|^\gamma}{\left(\kappa_m^T(s) \right)^\gamma} \diff s.
\]
Using the lower bound $\kappa_m^T(s)\geq s-\dt$, the upper bound $|s-\kappa_m^T(s)|\le \dt$ and a change of variables, one obtains the upper bounds
\[
\int_\dt^t \frac{|s-\kappa_m^T(s)|^\gamma}{\left( \kappa_m^T(s) \right)^\gamma} \diff s\le \int_\dt^t \frac{\dt^\gamma}{(s-\dt)^\gamma} \diff s \le {\dt^\gamma}\int_0^{t-\dt} \frac{1}{s^\gamma} \diff s
\le C_\gamma(T)\dt^\gamma.
\]
As a result, for the error term $\overline{\mathcal{A}_{3,2p}^m}(t)$, one obtains the following upper bound: there exists $C_{p,\gamma}(T)\in(0,\infty)$ such that one has
\begin{equation}\label{eq:errorA3}
\underset{t\in[0,T]}\sup~\overline{\mathcal{A}_{3,2p}^m}(t)\le C_{p,\gamma}(T)\bigl(1+\norm{u_0}_\infty\bigr)\dt^\gamma.
\end{equation}

\medskip

$\bullet$ Treatment of the error term $\overline{\mathcal{A}_{4,2p}^m}(t)$.

Applying the Minkowski inequality, for all $(t,x)\in[0,T]\times \overline{Q}$, one has
\[
\vvvert A_4^m(t,x) \vvvert_{2p} \le \int_0^t \int_Q G_d(t-\kappa_m^T(s),x,y) \vvvert b(\kappa_m^T(s),y,u(\kappa_m^T(s),y))-b(\kappa_m^T(s),y,u^m(\kappa_m^T(s),y))\vvvert_{2p} \diff y \diff s.
\]
Owing to the condition~\eqref{L} from Assumption~\ref{ass:coeff}, the mapping $u\mapsto b(s,y,u)$ satisfies a global Lipschitz continuity property. Therefore, recalling the definition of $\overline{\mathcal{E}}_{2p}^m(s)$ and applying the property~\eqref{eq:Gdint} of the heat kernel $G_d$ (see Subsection~\ref{sec:kernel}), one has
\begin{align*}
\vvvert A_4^m(t,x) \vvvert_{2p} &\le C\int_0^t \int_Q G_d(t-\kappa_m^T(s),x,y) \vvvert u(\kappa_m^T(s),y)-u^m(\kappa_m^T(s),y)\vvvert_{2p} \diff y \diff s\\
&\le C\int_0^t \int_Q G_d(t-\kappa_m^T(s),x,y) \overline{\mathcal{E}}_{2p}^m(\kappa_m^T(s)) \diff y \diff s\\
&\le C(T)\int_0^t \overline{\mathcal{E}}_{2p}^m(\kappa_m^T(s)) \diff s.
\end{align*}
As a result, for the error term $\overline{\mathcal{A}_{4,2p}^m}(t)$, one obtains the following upper bound: there exists $C(T)\in(0,\infty)$ such that one has
\begin{equation}\label{eq:errorA4}
\underset{t\in[0,T]}\sup~\overline{\mathcal{A}_{4,2p}^m}(t)\le C(T)\int_0^t \overline{\mathcal{E}}_{2p}^m(\kappa_m^T(s)) \diff s.
\end{equation}

This concludes the treatment of the error terms $\overline{\mathcal{A}_{j,2p}^m}(t)$ with $j\in\{1,\ldots,4\}$. It remains to deal with the error terms $\overline{\mathcal{B}_{j,2p}^m}(t)$ with $j\in\{1,\ldots,4\}$.

\medskip

$\bullet$ Treatment of the error term $\overline{\mathcal{B}_{1,2p}^m}(t)$.

Applying the Burkholder--Davis--Gundy inequality~\eqref{bdg}, for all $(t,x)\in[0,T]\times \overline{Q}$, one has
\[
\E[|B_1^m(t,x)|^{2p}]\le C_p(T)\E \left[\left|\int_0^t \|[G_d(t-s,x,\cdot)-G_d(t-\kappa_m^T(s),x,\cdot)]\sigma(s,\cdot,u(s,\cdot))\|_{(\alpha)}^2 \diff s\right|^p\right].
\]
For all $(t,x)\in[0,T]\times \overline{Q}$, introduce the auxiliary positive measure $\nu_{t,x}^{m}$ on $[0,t]\times Q^2$, which is absolutely continuous with respect to the Lebesgue measure, and with Radon--Nikodym derivative given by
\[
\frac{\diff \nu_{t,x}^{m}(s,y,z)}{\diff s \diff y \diff z}=\left|G_d(t-s,x,y)-G_d(t-\kappa_m^T(s),x,y)\right|\left|G_d(t-s,x,z)-G_d(t-\kappa_m^T(s),x,z)\right||y-z|^{-\alpha}.
\]
Making use of the auxiliary measure $\nu_{t,x}^{m}$ introduced above yields the upper bound
\[
\E[|B_1^m(t,x)|^{2p}]\le C_p(T)\E \left[ \left|\int_{0}^{t}\iint_{Q\times Q}|\sigma(s,y,u(s,y))| |\sigma(s,z,u(s,z))| \diff \nu_{t,x}^{m}(s,y,z) \right|^p \right].
\]
Applying the H\"older inequality, one then obtains
\begin{align*}
\E[|B_1^m(t,x)|^{2p}]&\le C_p(T)\int_{0}^{t}\iint_{Q\times Q}\E[|\sigma(s,y,u(s,y))|^p |\sigma(s,z,u(s,z))|^p] \diff \nu_{t,x}^{m}(s,y,z) \nu_{t,x}^{m}([0,t]\times Q^2)^{p-1}\\
&\le C_p(T)\underset{s\in[0,T]}\sup~\underset{y,z\in \overline{Q}}\sup~\E[|\sigma(s,y,u(s,y))|^p |\sigma(s,z,u(s,z))|^p]  \nu_{t,x}^{m}([0,t]\times Q^2)^{p}.
\end{align*}
Owing to the condition~\eqref{LG} from Assumption~\ref{ass:coeff}, the mapping $\sigma$ has at most linear growth. Therefore, applying the moment bounds~\eqref{eq:momo-u} from Proposition~\ref{prop:uExistence}, one obtains the upper bounds
\begin{align*}
\underset{s\in[0,T]}\sup~\underset{y,z\in \overline{Q}}\sup~\E[|\sigma(s,y,u(s,y))|^p |\sigma(s,z,u(s,z))|^p]&\le \underset{s\in[0,T]}\sup~\underset{y\in \overline{Q}}\sup~\E[|\sigma(s,y,u(s,y))|^{2p}]\\
&\le C_p(T)\bigl(1+\|u_0\|_{\infty}^{2p}\bigr).
\end{align*}
To proceed, one needs to prove an upper bound for
\[
\nu_{t,x}^{m}([0,t]\times Q^2)=\int_{0}^{t}\|G_d(t-s,x,\cdot)-G_d(t-\kappa_m^T(s),x,\cdot)\|_{(\alpha)}^{2} \diff s.
\]
For all $s\in(0,t)$, let $\delta(s)=(t-\kappa_m^T(s))-(t-s)=s-\kappa_m^T(s)$. In order to apply the inequality~\eqref{eq:GdEstimates-interpoled-strong} on the heat kernel from Lemma~\ref{lem:GdEstimates-interpoled} (see Subsection~\ref{sec:kernel}), the cases $\delta(s)<t-s$ and $\delta(s)\ge t-s$ are treated separately.

On the one hand, if $\delta(s)<t-s$, one obtains
\begin{align*}
\|G_d(t-s,x,\cdot)&-G_d(t-\kappa_m^T(s),x,\cdot)\|_{(\alpha)}^{2}\\
&\le C\frac{\delta(s)^2}{(t-s)^2}\iint_{Q\times Q}  (t-s)^{-d} e^{-c\frac{|y-x|^2}{2(t-s)}}
e^{-c\frac{|z-x|^2}{2(t-s)}} |y-z|^{-\alpha}\diff y\diff z.
\end{align*}
Under the condition $\delta(s)<t-s$, for all $\gamma\in(0,1)$ one has
\[
\frac{\delta(s)^2}{(t-s)^2}\le \frac{\delta(s)^{2\gamma}}{(t-s)^{2\gamma}}.
\]
Applying changes of variable $y'=y-x$ and $z'=z-x$, and $y''=y'/\sqrt{t-s}$ and $z''=z'/\sqrt{t-s}$, one has
\begin{align*}
\iint_{Q\times Q}  (t-s)^{-d} &e^{-c\frac{|y-x|^2}{2(t-s)}}
e^{-c\frac{|z-x|^2}{2(t-s)}} |y-z|^{-\alpha}\diff y\diff z\\
&\le \iint_{\R^d\times \R^d}  (t-s)^{-d} e^{-c\frac{|y'|^2}{2(t-s)}}
e^{-c\frac{|z'|^2}{2(t-s)}} |y'-z'|^{-\alpha}\diff y'\diff z'\\
&\le C(t-s)^{-\frac{\alpha}{2}}\iint_{\R^d\times \R^d}   e^{-c\frac{|y''|^2}{2}}
e^{-c\frac{|z''|^2}{2}} |y''-z''|^{-\alpha}\diff y''\diff z''\\
&\le C_\alpha (t-s)^{-\frac{\alpha}{2}}.
\end{align*}
On the other hand, if $\delta(s)=s-\kappa_m^T(s)\ge t-s$, applying the inequality~\eqref{eq:lem-aux1} from Lemma~\ref{lem:aux}, one obtains
\begin{align*}
\|G_d(t-s,x,\cdot)-G_d(t-\kappa_m^T(s),x,\cdot)\|_{(\alpha)}^{2}&\le \|G_d(t-s,x,\cdot)\|_{(\alpha)}^{2}+\|G_d(t-\kappa_m^T(s),x,\cdot)\|_{(\alpha)}^{2}\\
&\le C_\alpha (t-s)^{-\frac{\alpha}{2}}+C_\alpha (t-\kappa_m^T(s))^{-\frac{\alpha}{2}}\\
&\le C_\alpha(t-s)^{-\frac{\alpha}{2}}.
\end{align*}
Moreover, under the condition $\delta(s)\ge t-s$, for all $\gamma\in(0,1)$ one has
\[
1\le \frac{\delta(s)^{2\gamma}}{(t-s)^{2\gamma}}.
\]
As a result, in both cases one obtains the upper bound  
\[
\|G_d(t-s,x,\cdot)-G_d(t-\kappa_m^T(s),x,\cdot)\|_{(\alpha)}^{2}\le C_\alpha \frac{(s-\kappa_m^T(s))^{2\gamma}}{(t-s)^{2\gamma+\frac{\alpha}{2}}}.
\]
Therefore, taking into account the condition $\gamma\in(0,\frac12-\frac{\alpha}{4})$, for all $(t,x)\in[0,T]\times Q$, one has
\begin{align*}
\nu_{t,x}^{m}([0,t]\times Q^2)&=\int_{0}^{t}\|G_d(t-s,x,\cdot)-G_d(t-\kappa_m^T(s),x,\cdot)\|_{(\alpha)}^{2} \diff s\\
&\le C_\alpha \dt^{2\gamma}\int_0^t (t-s)^{-2\gamma-\frac\alpha2} \diff s\le C_{\alpha,\gamma}(T)\dt^{2\gamma}.
\end{align*}
Finally, for all $(t,x)\in[0,T]\times Q$, one has the upper bound
\[
\vvvert B_1^m(t,x)\vvvert_{2p}^{2p}=\E[|B_1^m(t,x)|^{2p}]\le C_{\alpha,\gamma}(T)\bigl(1+\|u_0\|_{\infty}^{2p}\bigr)\dt^{2p\gamma}.
\]
As a result, for the error term $\overline{\mathcal{B}_{1,2p}^m}(t)$, one obtains the following upper bound: there exists $C_{p,\alpha,\gamma}(T)\in(0,\infty)$ such that one has
\begin{equation}\label{eq:errorB1}
\underset{t\in[0,T]}\sup~\overline{\mathcal{B}_{1,2p}^m}(t)\le C_{p,\alpha,\gamma}(T)\bigl(1+\|u_0\|_{\infty}\bigr)\dt^{\gamma}.
\end{equation}

\medskip

$\bullet$ Treatment of the error term $\overline{\mathcal{B}_{2,2p}^m}(t)$.

Applying the inequality~\eqref{eq:lem-tech3} from Lemma~\ref{lem:tech} (with $\kappa(s)=\kappa_m^T(s)$), for all $(t,x)\in[0,T]\times \overline{Q}$, one has
\[
\E[|{B}_{2}^m(t,x)|^{2p}]\le C_{p,\alpha}(T)\underset{0\le s\le T}\sup~\underset{y\in Q}\sup~\E[|\sigma(s,y,u(s,y))-\sigma(\kappa_m^T(s),y,u(s,y))|^{2p}].
\]
Owing to the condition~\eqref{L} from Assumption~\ref{ass:coeff}, the mapping $s\in[0,T]\mapsto b(s,y,u)$ satisfies a H\"older continuity property with exponent $\frac12-\frac{\alpha}{4}$, uniformly with respect to $(y,u)\in\R^2$. Therefore, recalling that $\gamma\in(0,\frac12-\frac{\alpha}{4})$, one has
\[
\E[|{B}_{2}^m(t,x)|^{2p}]\le C_{p,\alpha}(T)\dt^{2p\gamma}.
\]
As a result, for the error term $\overline{\mathcal{B}_{2,2p}^m}(t)$, one obtains the following upper bound: there exists $C_{p,\alpha,\gamma}(T)\in(0,\infty)$ such that one has
\begin{equation}\label{eq:errorB2}
\underset{t\in[0,T]}\sup~\overline{\mathcal{B}_{2,2p}^m}(t)\le C_{p,\alpha,\gamma}(T)\dt^{\gamma}.
\end{equation}

\medskip

$\bullet$ Treatment of the error term $\overline{\mathcal{B}_{3,2p}^m}(t)$.

Applying the inequality~\eqref{eq:lem-tech1} from Lemma~\ref{lem:tech} (with $\kappa(s)=\kappa_m^T(s)$), for all $(t,x)\in[0,T]\times \overline{Q}$, one has
\[
\vvvert B_{3}^m(t,x) \vvvert_{2p}^2\le C_{p,\alpha}(T)\int_{0}^{t}(t-\kappa_m^T(s))^{-\frac{\alpha}{2}}\underset{y\in \overline{Q}}\sup~\vvvert\sigma(\kappa_m^T(s),y,u(s,y))-\sigma(\kappa_m^T(s),y,u(\kappa_m^T(s),y))\vvvert_{2p}^2\diff s.
\]
Owing to the condition~\eqref{L} from Assumption~\ref{ass:coeff}, the mapping $\sigma$ satisfies a global Lipschitz continuity property. Therefore, one has
\[
\vvvert B_{3}^m(t,x) \vvvert_{2p}^2\le C_{p,\alpha}(T)\int_{0}^{t}(t-\kappa_m^T(s))^{-\frac{\alpha}{2}}\underset{y\in \overline{Q}}\sup~\vvvert u(s,y)-u(\kappa_m^T(s),y)\vvvert_{2p}^2\diff s.
\]
Like in the proof of the upper bound~\eqref{eq:errorA3} for the error term $\overline{\mathcal{B}_{3,2p}^m}(t)$, the regularity property~\eqref{eq:reg} from Proposition~\ref{prop:reg} is not directly applicable, and it is necessary to decompose the integral as
\begin{align*}
\int_{0}^{t}(t-\kappa_m^T(s))^{-\frac{\alpha}{2}}&\underset{y\in \overline{Q}}\sup~\vvvert u(s,y)-u(\kappa_m^T(s),y)\vvvert_{2p}^2\diff s\\
&=\int_{0}^{\min(t,\dt)}(t-\kappa_m^T(s))^{-\frac{\alpha}{2}}\underset{y\in \overline{Q}}\sup~\vvvert u(s,y)-u(\kappa_m^T(s),y)\vvvert_{2p}^2\diff s\\
&+\mathds{1}_{t>\dt}\int_{\dt}^{t}(t-\kappa_m^T(s))^{-\frac{\alpha}{2}}\underset{y\in \overline{Q}}\sup~\vvvert u(s,y)-u(\kappa_m^T(s),y)\vvvert_{2p}^2\diff s.
\end{align*}
On the one hand, note that for all $s\in[0,\min(t,\dt))$ one has $u(\kappa_m^T(s),\cdot)=u_0$, thus applying the Minkowski inequality and the moment bounds~\eqref{eq:momo-u} from Proposition~\ref{prop:uExistence}, for the first integral one has
\begin{align*}
\int_{0}^{\min(t,\dt)}(t-\kappa_m^T(s))^{-\frac{\alpha}{2}}&\underset{y\in \overline{Q}}\sup~\vvvert u(s,y)-u(\kappa_m^T(s),y)\vvvert_{2p}^2\diff s\\
&\le \int_{0}^{\min(t,\dt)}(t-\kappa_m^T(s))^{-\frac{\alpha}{2}}\bigl(\underset{y\in \overline{Q}}\sup~\vvvert u(s,y)\vvvert_{2p}^2+\norm{u_0}_\infty^2\bigr)\diff s\\
&\le \int_{0}^{\min(t,\dt)}(t-\kappa_m^T(s))^{-\frac{\alpha}{2}}\diff s\bigl(1+\norm{u_0}_\infty\bigr)^2.
\end{align*}
To deal with the integral in the right-hand side above, observe that since $\kappa_m^T(s)=0$ for all $s\in[0,\min(t,\dt)]$, one has
\[
\int_0^{\min(t,\dt)} (t-\kappa_m^T(s))^{-\frac{\alpha}{2}} \diff s=\min(t,\dt) t^{-\frac{\alpha}{2}} =
\min(t,\dt)^{\frac{\alpha}{2}}\min(t,\dt)^{1-\frac{\alpha}{2}} t^{-\frac{\alpha}{2}}\le \dt^{1-\frac{\alpha}{2}}\le C\tau^{2\gamma}.
\]
On the other hand, applying the regularity property~\eqref{eq:reg} from Proposition~\ref{prop:reg}, for all $t\in(\tau,\infty)$, one has
\begin{align*}
\int_{\dt}^{t}(t-\kappa_m^T(s))^{-\frac{\alpha}{2}}&\underset{y\in \overline{Q}}\sup~\vvvert u(s,y)-u(\kappa_m^T(s),y)\vvvert_{2p}^2\diff s\\
&\le C_{p,\gamma}(T) \int_\dt^t (t-\kappa_m^T(s))^{-\frac{\alpha}{2}}\frac{|s-\kappa_m^T(s)|^{2\gamma}}{\kappa_m^T(s)^{2\gamma}} \diff s (1+\|u_0\|_\infty)^2.
\end{align*}
The integral in the right-hand side above can be treated as follows: using the inequalities $s\ge \kappa_m^T(s)$ and $\kappa_m^T(s)\ge s-\dt$, for all $t>\dt$ one has
\begin{align*}
\int_\dt^t (t-\kappa_m^T(s))^{-\frac{\alpha}{2}}\frac{|s-\kappa_m^T(s)|^{2\gamma}}{\kappa_m^T(s)^{2\gamma}} \diff s
&\le\dt^{2\gamma} \int_\dt^t (t-\kappa_m^T(s))^{-\frac{\alpha}{2}}\kappa_m^T(s)^{-2\gamma} \diff s \\
&\le\dt^{2\gamma} \int_\dt^t (t-s)^{-\frac{\alpha}{2}}(s-\dt)^{-2\gamma} \diff s \\
&\le\dt^{2\gamma} \int_0^{t-\dt} (t-\dt-s)^{-\frac{\alpha}{2}}s^{-2\gamma} \diff s \\
&\le\dt^{2\gamma} (t-\dt)^{1-\frac{\alpha}2-2\gamma} \int_0^1 (1-s')^{-\frac{\alpha}{2}}(s')^{-2\gamma} \diff s'\\
&\le C_{\alpha,\gamma}(T)\dt^{2\gamma}.
\end{align*}
As a result, for the error term $\overline{\mathcal{B}_{3,2p}^m}(t)$, one obtains the following upper bound: there exists $C_{p,\alpha,\gamma}(T)\in(0,\infty)$ such that one has
\begin{equation}\label{eq:errorB3}
\underset{t\in[0,T]}\sup~\overline{\mathcal{B}_{3,2p}^m}(t)\le C_{p,\alpha,\gamma}(T)\bigl(1+\|u_0\|_{\infty}\bigr)\dt^{\gamma}.
\end{equation}

\medskip

$\bullet$ Treatment of the error term $\overline{\mathcal{B}_{4,2p}^m}(t)$.

Applying the inequality~\eqref{eq:lem-tech2} from Lemma~\ref{lem:tech} (with $\kappa(s)=\kappa_m^T(s)\leq s$), for all $(t,x)\in[0,T]\times \overline{Q}$, one has
\[
\E[|B_{4}^m(t,x)|^{2p}]\le C_{p,\alpha}(T)\int_{0}^{t}(t-s)^{-\frac{\alpha}{2}}\ {\sup_{y \in \overline{Q}}}\ \E[|\sigma(\kappa_m^T(s),y,u(\kappa_m^T(s),y))-\sigma(\kappa_m^T(s),y,u^m(\kappa_m^T(s),y))|^{2p}]\diff s.
\]
Owing to the condition~\eqref{L} from Assumption~\ref{ass:coeff}, the mapping $\sigma$ satisfies a global Lipschitz continuity property. Therefore, recalling the definition of $\overline{\mathcal{E}}_{2p}^m(s)$, one has
\begin{align*}
\E[|B_{4}^m(t,x)|^{2p}]&\le C_{p,\alpha}(T)\int_{0}^{t}(t-s)^{-\frac{\alpha}{2}}\ {\sup_{y \in \overline{Q}}}\ \E[|u(\kappa_m^T(s),y)-u^m(\kappa_m^T(s),y)|^{2p}]\diff s\\
&\le C_{p,\alpha}(T)\int_0^t (t-s)^{-\frac{\alpha}{2}}\overline{\mathcal{E}}_{2p}^m(\kappa_m^T(s))^{2p} \diff s.
\end{align*}
As a result, for the error term $\overline{\mathcal{B}_{4,2p}^m}(t)$, one obtains the following upper bound: there exists $C(T)\in(0,\infty)$ such that one has
\begin{equation}\label{eq:errorB4}
\underset{t\in[0,T]}\sup~\overline{\mathcal{B}_{4,2p}^m}(t)^{2p}\le C(T)\int_0^t (t-s)^{-\frac{\alpha}{2}} \overline{\mathcal{E}}_{2p}^m(\kappa_m^T(s))^{2p} \diff s.
\end{equation}

This concludes the treatment of the error terms $\overline{\mathcal{B}_{j,2p}^m}(t)$ for $j\in\{1,\ldots,4\}$.

\medskip

$\bullet$ Conclusion.

Gathering the upper bounds~\eqref{eq:errorA1},~\eqref{eq:errorA2},~\eqref{eq:errorA3} and~\eqref{eq:errorA4} for the error terms $\overline{\mathcal{A}_{j,2p}^m}(t)$ for $j\in\{1,\ldots,4\}$, and the upper bounds~\eqref{eq:errorB1},~\eqref{eq:errorB2},~\eqref{eq:errorB3} and~\eqref{eq:errorB4} for the error terms $\overline{\mathcal{B}_{j,2p}^m}(t)$ for $j\in\{1,\ldots,4\}$, one obtains the following inequality: for all $T\in(0,\infty)$, $\alpha\in(0,2\wedge d)$, $\gamma\in(\frac12-\frac{\alpha}{4})$ and $p\in[1,\infty)$, there exists $C_{p,\alpha,\gamma}(T)\in(0,\infty)$ such that for all $t\in[0,T]$ one has
\begin{equation}
\overline{\mathcal{E}}^m_{2p}(t)^{2p}\le C_{p,\alpha,\gamma}(T)\dt^{2p\gamma}\bigl(1+\norm{u_0}_\infty\bigr)^{2p}
+C_{p,\alpha,\gamma}(T)\int_0^t \bigl( 1+(t-s)^{-\frac\alpha2}\bigr) \overline{\mathcal{E}}_{2p}^m(\kappa_m^T(s))^{2p}\diff s.
\end{equation}
Applying a version of the Gr\"onwall inequality (c.f \cite[Lem. 15]{Dalang99})  yields the upper bound
\begin{equation}
\underset{t\in[0,T]}\sup~\overline{\mathcal{E}}^m_{2p}(t)^{2p}\le C_{p,\alpha,\gamma}(T)\dt^{2p\gamma}\bigl(1+\norm{u_0}_\infty\bigr)^{2p}.
\end{equation}
This provides the strong error estimates~\eqref{eq:main} and the proof of Theorem~\ref{thm:main} is completed.
\end{proof}

\begin{appendix}\label{app}

\section{Proof of the properties of the heat kernel}\label{app:aux}

The objective of this section is to provide the proofs of the properties of the heat kernel stated in Subsection~\ref{sec:kernel}.

Let us first provide the proof of Lemma~\ref{lem:GdEstimates-interpoled}.
\begin{proof}[Proof of Lemma~\ref{lem:GdEstimates-interpoled}]
Let $t_1<t_2$. The proof of the inequality~\eqref{eq:GdEstimates-interpoled-strong} proceeds by considering two cases.

First, assume that $\frac{t_2-t_1}{t_1}< 1$, i.\,e. that $t_1<t_2<2t_1$.
Applying the inequality~\eqref{eq:GdEstimates-dtG} on the temporal derivative of the heat kernel $G_d$, one obtains
\[
|G_d(t_2,x,y)-G_d(t_1,x,y)|\le \int_{t_1}^{t_2}|\partial_tG_d(t,x,y)|\diff t\le C_d(t_2-t_1)t_1^{-\frac{d}{2}-1}e^{-c_d\frac{|y-x|^2}{t_2}}.
\]
Under the condition $t_2<2t_1$, one obtains the upper bound
\[
|G_d(t_2,x,y)-G_d(t_1,x,y)|\le C_d\frac{(t_2-t_1)}{t_1}t_1^{-\frac{d}{2}}e^{-c_d\frac{|y-x|^2}{2t_1}}.
\]
Secondly, assume that $\frac{t_2-t_1}{t_1}\ge 1$, i.\,e. that $t_2\ge 2t_1$. Applying the triangle inequality and the inequality~\eqref{eq:GdEstimates-G} from Lemma~\ref{lem:GdEstimates}, one obtains for all $x,y\in \overline{Q}$
\begin{align*}
|G_d(t_2,x,y)-G_d(t_1,x,y)|&\le G_d(t_2,x,y)+G_d(t_1,x,y)\\
&\le C_d\bigl(t_1^{-\frac{d}{2}}e^{-c_d\frac{|y-x|^2}{t_1}}+t_2^{-\frac{d}{2}}e^{-c_d\frac{|y-x|^2}{t_2}}\bigr).
\end{align*}
Combining the results above then provides the inequality~\eqref{eq:GdEstimates-interpoled-strong}.

The inequality~\eqref{eq:GdEstimates-interpoled} is a straightforward consequence of the inequality~\eqref{eq:GdEstimates-interpoled-strong}. Indeed, one first obtains the inequality
\[
|G_d(t_2,x,y)-G_d(t_1,x,y)|\le C_d\min\left(\frac{t_2-t_1}{t_1},1\right)\left(t_1^{-\frac{d}{2}}e^{-c_d\frac{|y-x|^2}{2t_1}}+t_2^{-\frac{d}{2}}e^{-c_d\frac{|y-x|^2}{2t_2}}\right)
\]
and if $\gamma\in(0,1)$ for all $x\in\R$ one has $\min(x,1)\le x^\gamma$.

This concludes the proof of Lemma~\ref{lem:GdEstimates-interpoled}.
\end{proof}

Let us now provide the proof of Lemma~\ref{lem:aux}. The inequalities~\eqref{eq:lem-aux1},~\eqref{eq:lem-aux2},~\eqref{eq:lem-aux3} and~\eqref{eq:lem-aux4} are established successively. Recall that $d\in\N$ and $\alpha\in(0,2\wedge d)$.

\begin{proof}[Proof of the inequality~\eqref{eq:lem-aux1}]
Let $t\in(0,\infty)$ and $x\in \overline{Q}$. Applying the inequality~\eqref{eq:GdEstimates-G}, one has
\begin{align*}
\|G_d(t,x,\cdot)\|_{(\alpha)}^{2}&=\iint_{Q\times Q} G_d(t,x,y)G_d(t,x,z)|y-z|^{-\alpha} \diff y \diff z\\
&\le Ct^{-d}\iint_{Q\times Q} e^{-c\frac{|y-x|^2+|z-x|^2}{t}}|y-z|^{-\alpha} \diff y \diff z\\
&\le \mathcal{G}_{t,x}^1+\mathcal{G}_{t,x}^2,
\end{align*}
where $\mathcal{G}_{t,x}^1$ and $\mathcal{G}_{t,x}^2$ are defined as
\begin{align*}
\mathcal{G}_{t,x}^1&=Ct^{-d}\iint_{Q\times Q} \mathds{1}_{|x-y|\ge |y-z|}e^{-c\frac{|y-x|^2+|z-x|^2}{t}}|y-z|^{-\alpha} \diff y \diff z\\
\mathcal{G}_{t,x}^2&=Ct^{-d}\iint_{Q\times Q} \mathds{1}_{|x-y|\le |y-z|}e^{-c\frac{|y-x|^2+|z-x|^2}{t}}|y-z|^{-\alpha} \diff y \diff z.
\end{align*}

Concerning $\mathcal{G}_{t,x}^1$, taking into account the condition $|x-y|\ge |y-z|$ and applying the change of variables $(y',z')=(y-z,z-x)$, one obtains
\begin{align*}
\mathcal{G}_{t,x}^1&=C t^{-d}\iint_{Q\times Q} \mathds{1}_{|x-y|\ge |y-z|} e^{-c\frac{|y-x|^2}{t}}e^{-c\frac{|z-x|^2}{t}} |y-z|^{-\alpha} \diff y \diff z\\
&\le C t^{-d}\iint_{Q\times Q}  e^{-c\frac{|y-z|^2}{t}}e^{-c\frac{|z-x|^2}{t}} |y-z|^{-\alpha} \diff y \diff z\\
&\le Ct^{-d}\int_{\R^d} e^{-c\frac{|y'|^2}{t}} |y'|^{-\alpha} \diff y' \int_{\R^d} e^{-c\frac{|z'|^2}{t}} \diff z'.
\end{align*}

Concerning $\mathcal{G}_{t,x}^2$, taking into account the condition $|x-y|\le |y-z|$ and applying the change of variables $(y',z')=(y-x,z-x)$, one obtains
\begin{align*}
\mathcal{G}_{t,x}^2&=C t^{-d}\iint_{Q\times Q} \mathds{1}_{|x-y|\le |y-z|} e^{-c\frac{|y-x|^2}{t}}e^{-c\frac{|z-x|^2}{t}} |y-z|^{-\alpha} \diff y \diff z\\
&\le C t^{-d}\iint_{Q\times Q}  e^{-c\frac{|y-x|^2}{t}}e^{-c\frac{|z-x|^2}{t}} |y-x|^{-\alpha} \diff y \diff z\\
&\le Ct^{-d}\int_{\R^d} e^{-c\frac{|y'|^2}{t}} |y'|^{-\alpha} \diff y' \int_{\R^d} e^{-c\frac{|z'|^2}{t}} \diff z'.
\end{align*}
Note that one obtains the same upper bound for the two terms $\mathcal{G}_{t,x}^1$ and $\mathcal{G}_{t,x}^2$. Applying the supplementary change of variables $(y'',z'')=t^{-\frac12}(y',z')$, one then has
\begin{align*}
\mathcal{G}_{t,x}^1+\mathcal{G}_{t,x}^2&\le Ct^{-d}\int_{\R^d} e^{-c\frac{|y'|^2}{t}} |y'|^{-\alpha} \diff y' \int_{\R^d} e^{-c\frac{|z'|^2}{t}} \diff z'\\
&=Ct^{-\frac{\alpha}{2}}\int_{\R^d} |y''|^{-\alpha}e^{-c|y''|^2} \diff y'' \int_{\R^d} e^{-c|z''|^2} \diff z''\\
&\le C_\alpha t^{-\frac{\alpha}{2}}.
\end{align*}
One thus obtains the following upper bound: there exists $C_\alpha\in(0,\infty)$ such that for all $t\in(0,\infty)$ one has
\begin{equation}\label{eq:AAUX}
t^{-d}\iint_{Q\times Q} e^{-c\frac{|y-x|^2+|z-x|^2}{t}}|y-z|^{-\alpha} \diff y \diff z \le C_\alpha t^{-\frac{\alpha}{2}}.
\end{equation}
As a result, one has
\[
\|G_d(t,x,\cdot)\|_{(\alpha)}^{2}\le C_\alpha t^{-\frac{\alpha}{2}}.
\]
The proof of the inequality~\eqref{eq:lem-aux1} is thus completed.
\end{proof}

\begin{proof}[Proof of the inequality~\eqref{eq:lem-aux2}]
Let $t_2\ge t_1\ge 0$ and $x\in \overline{Q}$. Applying the change of variable $s=t_2-t$ and the inequality~\eqref{eq:lem-aux1}, one has
\[
\int_{t_1}^{t_2}\|G_d(t_2-t,x,\cdot)\|_{(\alpha)}^2 \diff t=\int_{0}^{t_2-t_1}\|G_d(s,x,\cdot)\|_{(\alpha)}^2 \diff s\le C_\alpha\int_{0}^{t_2-t_1}s^{-\frac{\alpha}{2}} \diff s\le C_\alpha |t_2-t_1|^{1-\frac{\alpha}{2}}.
\]
The proof of the inequality~\eqref{eq:lem-aux2} is thus completed.
\end{proof}

\begin{proof}[Proof of the inequality~\eqref{eq:lem-aux3}]
Let $x_1,x_2\in \overline{Q}$, and set $\delta x=x_2-x_1$. One has the decomposition
\[
\int_{0}^{+\infty}\|G_d(t,x_2,\cdot)-G_d(t,x_1,\cdot)\|_{(\alpha)}^{2}\diff t=\delta \mathcal{G}^1+\delta \mathcal{G}^2,
\]
where $\delta \mathcal{G}^1$ and $\delta \mathcal{G}^2$ are defined as
\begin{align*}
\delta \mathcal{G}^1&=\int_{0}^{|\delta x|^2}\|G_d(t,x_2,\cdot)-G_d(t,x_1,\cdot)\|_{(\alpha)}^{2} \diff t\\
\delta \mathcal{G}^2&=\int_{|\delta x|^2}^{\infty}\|G_d(t,x_2,\cdot)-G_d(t,x_1,\cdot)\|_{(\alpha)}^{2} \diff t.
\end{align*}
Let us first deal with $\delta \mathcal{G}^1$. One has
\begin{align*}
\int_{0}^{|\delta x|^2}\|G_d(t,x_2,\cdot)-G_d(t,x_1,\cdot)\|_{(\alpha)}^{2} \diff t&\le 2\int_{0}^{|\delta x|^2}\|G_d(t,x_1,\cdot)\|_{(\alpha)}^{2} \diff t+2\int_{0}^{|\delta x|^2}\|G_d(t,x_2,\cdot)\|_{(\alpha)}^{2} \diff t\\
&\le 4\int_{0}^{|\delta x|^2}\underset{x\in Q}\sup~\|G_d(t,x,\cdot)\|_{(\alpha)}^{2} \diff t.
\end{align*}
Applying the inequality~\eqref{eq:lem-aux1}, one then gets
\[
\int_{0}^{|\delta x|^2}\underset{x\in Q}\sup~\|G_d(t,x,\cdot)\|_{(\alpha)}^{2} \diff t\le C_\alpha \int_{0}^{|\delta x|^2}t^{-\frac{\alpha}{2}} \diff t=C_\alpha (|\delta x|^2)^{1-\frac{\alpha}{2}}=C_\alpha |\delta x|^{2-\alpha}.
\]
As a result, one obtains the upper bound
\begin{equation}\label{eq:deltaGaux1}
\delta \mathcal{G}^1\le C_\alpha |\delta x|^{2-\alpha}.
\end{equation}
Let us now deal with $\delta \mathcal{G}^2$. By the definition of the norm $\|\cdot\|_{(\alpha)}$, one has
\begin{align*}
\delta \mathcal{G}^2&=\int_{|\delta x|^2}^{\infty}\|G_d(t,x_2,\cdot)-G_d(t,x_1,\cdot)\|_{(\alpha)}^{2} \diff t\\
&=\int_{|\delta x|^2}^{\infty}\iint_{Q\times Q}|G_d(t,x_2,y)-G_d(t,x_1,y)||G_d(t,x_2,z)-G_d(t,x_1,z)| |y-z|^{-\alpha} \diff y \diff z \diff t.
\end{align*}
Let $x(\tau)=(1-\tau)x_1+\tau x_2$ for all $\tau\in[0,1]$. For all $t\ge 0$ and all $y,z\in \overline{Q}$, one obtains
\begin{align*}
|G_d(t,x_2,y)-G_d(t,x_1,y)|=|G_d(t,x(1),y)-G_d(t,x(0),y)|&=\big|\int_{0}^{1} \nabla_x G_d(t,x(\theta),y)\cdot \delta x \diff \theta\big|\\
&\le |\delta x|\int_{0}^{1} |\nabla_x G_d(t,x(\theta),y)| \diff \theta,\\
|G_d(t,x_2,z)-G_d(t,x_1,z)|=|G_d(t,x(1),z)-G_d(t,x(0),z)|&=\big|\int_{0}^{1} \nabla_x G_d(t,x(\eta),z)\cdot \delta x \diff \eta\big|\\
&\le |\delta x|\int_{0}^{1} |\nabla_x G_d(t,x(\eta),z)| \diff \eta.
\end{align*}
As a result, applying the inequality~\eqref{eq:GdEstimates-dxG}, one has the estimates
\begin{align*}
\delta \mathcal{G}^2
&\le |\delta x|^2\int_{|\delta x|^2}^{\infty}\iint_{Q\times Q} \iint_{[0,1]^2} |\nabla_xG_d(t,x(\theta),y)||\nabla_xG_d(t,x(\eta),y)| |y-z|^{-\alpha} \diff y \diff z \diff\theta \diff\eta \diff t\\
&\le C|\delta x|^2\int_{|\delta x|^2}^{\infty} t^{-(d+1)}\iint_{[0,1]^2} \iint_{Q\times Q}  e^{-c\frac{|y-x(\theta)|^2+|z-x(\eta)|^2}{t}} |y-z|^{-\alpha} \diff y \diff z \diff \theta \diff \eta \diff t\\
&\le \delta \mathcal{G}^{2,1}+\delta \mathcal{G}^{2,2},
\end{align*}
where $\delta \mathcal{G}^{2,1}$ and $\delta \mathcal{G}^{2,2}$ are defined as
\begin{align*}
\delta \mathcal{G}^{2,1}&=C|\delta x|^2\int_{|\delta x|^2}^{\infty} t^{-(d+1)}\iint_{[0,1]^2} \iint_{Q\times Q} \mathds{1}_{|y-x(\theta)|\ge |y-z|}e^{-c\frac{|y-x(\theta)|^2+|z-x(\eta)|^2}{t}} |y-z|^{-\alpha}  \diff y \diff z \diff \theta \diff \eta \diff t,\\
\delta \mathcal{G}^{2,2}&=C|\delta x|^2\int_{|\delta x|^2}^{\infty} t^{-(d+1)}\iint_{[0,1]^2} \iint_{Q\times Q} \mathds{1}_{|y-x(\theta)|\le |y-z|}e^{-c\frac{|y-x(\theta)|^2+|z-x(\eta)|^2}{t}} |y-z|^{-\alpha} \diff y \diff z \diff \theta \diff \eta \diff t.
\end{align*}
Concerning $\delta \mathcal{G}^{2,1}$, taking into account the condition $|y-x(\theta)|\ge |y-z|$ and applying the change of variable $(y',z')=(y-z,z-x(\eta))$, one obtains
\begin{align*}
\delta \mathcal{G}^{2,1}
&\le C|\delta x|^2\int_{|\delta x|^2}^{\infty} t^{-(d+1)}\iint_{[0,1]^2} \iint_{Q\times Q} e^{-c\frac{|y-z|^2}{t}} e^{-c\frac{|z-x(\eta)|^2}{t}} |y-z|^{-\alpha} \diff y \diff z \diff \theta \diff \eta \diff t\\
&\le C|\delta x|^2\int_{|\delta x|^2}^{\infty} t^{-(d+1)}\iint_{[0,1]^2} \iint_{\R^d\times \R^d} e^{-c\frac{|y'|^2}{t}} e^{-c\frac{|z'|^2}{t}} |y'|^{-\alpha} \diff y' \diff z' \diff \theta \diff \eta \diff t\\
&\le C|\delta x|^2\int_{|\delta x|^2}^{\infty} t^{-(d+1)} \iint_{\R^d\times \R^d} e^{-c\frac{|y'|^2}{t}} e^{-c\frac{|z'|^2}{t}} |y'|^{-\alpha} \diff y' \diff z' \diff t.
\end{align*}
Concerning $\delta \mathcal{G}^{2,2}$, taking into account the condition $|y-x(\theta)|\le |y-z|$ and applying the change of variable $(y',z')=(y-x(\theta),z-x(\eta))$, one obtains
\begin{align*}
\delta \mathcal{G}^{2,2}
&\le C|\delta x|^2\int_{|\delta x|^2}^{\infty}t^{-(d+1)}\iint_{[0,1]^2} \iint_{Q\times Q} \mathds{1}_{|y-x(\theta)|\le |y-z|}e^{-c\frac{|y-x(\theta)|^2}{t}} e^{-c\frac{|z-x(\eta)|^2}{t}} |y-x(\theta)|^{-\alpha} \diff y \diff z \diff \theta \diff \eta \diff t\\
&\le C|\delta x|^2\int_{|\delta x|^2}^{\infty} t^{-(d+1)}\iint_{[0,1]^2} \iint_{\R^d\times \R^d} e^{-c\frac{|y'|^2}{t}} e^{-c\frac{|z'|^2}{t}} |y'|^{-\alpha} \diff y' \diff z' \diff \theta \diff \eta \diff t\\
&\le C|\delta x|^2\int_{|\delta x|^2}^{\infty} t^{-(d+1)} \iint_{\R^d\times \R^d} e^{-c\frac{|y'|^2}{t}} e^{-c\frac{|z'|^2}{t}} |y'|^{-\alpha} \diff y' \diff z' \diff t.
\end{align*}
Observe that one gets the same upper bound for the two terms $\delta\mathcal{G}^{2,1}$ and $\delta \mathcal{G}^{2,2}$. Applying the supplementary change of variables $(y'',z'')=t^{-\frac12}(y',z')$, one then obtains the upper bounds
\begin{align*}
\delta \mathcal{G}^2&\le \delta \mathcal{G}^{2,1}+\mathcal{G}^{2,2}\\
&\le C|\delta x|^2\int_{|\delta x|^2}^{\infty} t^{-1-\frac{\alpha}{2}} \diff t \int_{\R^d} |y''|^{-\alpha}e^{-c|y''|^2} \diff y'' \int_{\R^d} e^{-c|z''|^2} \diff z''\\
&\le C_\alpha|\delta x|^2 (|\delta x|^2)^{-\frac{\alpha}{2}}=C_\alpha|\delta x|^{2-\alpha}.
\end{align*}
As a result, one gets the upper bound
\begin{equation}\label{eq:deltaGaux2}
\delta \mathcal{G}^2\le C_\alpha |\delta x|^{2-\alpha}.
\end{equation}
Combining the upper bounds~\eqref{eq:deltaGaux1} and~\eqref{eq:deltaGaux2} yields the inequality
\[
\int_{0}^{+\infty}\|G_d(t,x_2,\cdot)-G_d(t,x_1,\cdot)\|_{(\alpha)}^{2}\diff t=\delta \mathcal{G}^1+\delta \mathcal{G}^2 \le C_\alpha |\delta x|^{2-\alpha}.
\]
The proof of the inequality~\eqref{eq:lem-aux3} is thus completed.
\end{proof}

\begin{proof}[Proof of the inequality~\eqref{eq:lem-aux4}]
Let $t_2\ge t_1\ge 0$ and $x\in \overline{Q}$.

Owing to the definition of the norm $\|\cdot\|_{(\alpha)}$, one has
\begin{align*}
\int_{0}^{t_1}&\|G_d(t_2-t,x,\cdot)-G_d(t_1-t,x,\cdot)\|_{(\alpha)}^{2}\diff t\\
&=\int_{0}^{t_1}\iint_{Q\times Q} |G_d(t_2-t,x,y)-G_d(t_1-t,x,y)||G_d(t_2-t,x,z)-G_d(t_1-t,x,z)| |y-z|^{-\alpha} \diff y \diff z \diff t.
\end{align*}
Let $\gamma\in(0,1)$. Applying the inequality~\eqref{eq:GdEstimates-interpoled} on the heat kernel $G_d$ (see Subsection~\ref{sec:kernel}),
there exists $C_\gamma\in(0,\infty)$ such that for all $y\in \overline{Q}$ and all $t\in(0,t_1)$ one has
\begin{align*}
|G_d(t_2-t,x,y)&-G_d(t_1-t,x,y)|\\
&\le \bigl(|G_d(t_2-t,x,y)-G_d(t_1-t,x,y)|\bigr)^{\gamma}\bigl(G_d(t_2-t,x,y)+G_d(t_1-t,x,y)\bigr)^{1-\gamma}\\
&\le C_\gamma(t_1-t)^{-\frac{d}{2}-\gamma}|t_2-t_1|^{\gamma}e^{-c\frac{|y-x|^2}{t_1-t}}.
\end{align*}
Applying the change of variables $(y',z')=(y-x,z-x)$, one has the estimates
\begin{align*}
\int_{0}^{t_1}&\|G(t_2-t,x,\cdot)-G(t_1-t,x,\cdot)\|_{(\alpha)}^{2}\diff t\\
&\le C_\gamma|t_2-t_1|^{2\gamma}\int_{0}^{t_1} (t_1-t)^{-d-2\gamma} 
\int_{\R^d\times\R^d}e^{-c\gamma\frac{|y-x|^2}{t_1-t}} e^{-c\gamma\frac{|z-x|^2}{t_1-t}}|y-z|^{-\alpha} \diff y \diff z \diff t.
\end{align*}
Applying the auxiliary inequality~\eqref{eq:AAUX} from the proof of the inequality~\eqref{eq:lem-aux1}, one then obtains
\begin{align*}
\int_{0}^{t_1}\|G(t_2-t,x,\cdot)-G(t_1-t,x,\cdot)\|_{(\alpha)}^{2}\diff t&\le C_\alpha|t_2-t_1|^{2\gamma}\int_{0}^{t_1}(t_1-t)^{-2\gamma-\frac{\alpha}{2}}\diff t.
\end{align*}
Assuming that $\gamma\in(0,1-\frac{\alpha}{2})$, one has $2\gamma+\frac{\alpha}{2}<1$, and thus one gets the relation
\[
\int_{0}^{t_1}(t_1-t)^{-2\gamma-\frac{\alpha}{2}}dt=\int_{0}^{t_1}t^{-2\gamma-\frac{\alpha}{2}}\diff t=C_{\alpha,\gamma}t_1^{1-2\gamma-\frac{\alpha}{2}}.
\]
As a result, there exists $C_{\alpha,\gamma}\in(0,\infty)$ such that one has the bound
\[
\int_{0}^{t_1}\|G(t_2-t,x,\cdot)-G(t_1-t,x,\cdot)\|_{(\alpha)}^{2}\diff t \le C_{\alpha,\gamma}t_1^{1-2\gamma-\frac{\alpha}{2}}|t_2-t_1|^{2\gamma}.
\]
The proof of the inequality~\eqref{eq:lem-aux4} is thus completed.
\end{proof}

It remains to provide the proof of Lemma~\ref{lem:tech}.
\begin{proof}[Proof of Lemma~\ref{lem:tech}]
Applying the Burkholder--Davis--Gundy inequality~\eqref{bdg}, one has the upper bound
\[
\E\left[\left|\int_{0}^{t}\int_Q G_d(t-\kappa(s),x,y)X(s,y)\Falpha(\diff s,\diff y)\right|^{2p} \right]\le
C_p(T)\E \left[\left|\int_0^t \|G_d(t-\kappa(s),x,\cdot)X(s,\cdot)\|_{(\alpha)}^2 \diff s\right|^p \right].
\]
Owing to the definition of the norm $\vvvert \cdot \vvvert_{2p}$, one obtains
\begin{align*}
\vvvert\int_{0}^{t}\int_Q G_d(t-\kappa(s),x,y)X(s,y)\Falpha(\diff s,\diff y)\vvvert_{2p}&=
\left(\E \left[\left|\int_{0}^{t}\int_Q G_d(t-\kappa(s),x,y)X(s,y)\Falpha(\diff s,\diff y)\right|^{2p} \right]\right)^{\frac{1}{2p}}\\
&\le C_p(T)\left(\vvvert \int_0^t \|G_d(t-\kappa(s),x,\cdot)X(s,\cdot)\|_{(\alpha)}^2 \diff s \vvvert_p\right)^{\frac12}.
\end{align*}
Applying the Minkowski inequality for the norm $\vvvert \cdot \vvvert_p$ and recalling the definition of the norm $\|\cdot\|_{(\alpha)}$, one obtains
\begin{align*}
\vvvert \int_0^t &\|G_d(t-\kappa(s),x,\cdot)X(s,\cdot)\|_{(\alpha)}^2 \diff s \vvvert_p \le \int_0^t \vvvert\|G_d(t-\kappa(s),x,\cdot)X(s,\cdot)\|_{(\alpha)}^2 \vvvert_p \diff s \\
&\le \int_0^t \iint_{Q\times Q} G_d(t-\kappa(s),x,y) G_d(t-\kappa(s),x,z)\vvvert X(s,y)X(s,z) \vvvert_p |y-z|^{-\alpha} \diff y \diff z \diff s.
\end{align*}
Note that by the Cauchy--Schwarz inequality, for all $s\ge 0$, one has
\[
\underset{y,z\in \overline{Q}}\sup~\vvvert X(s,y)X(s,z) \vvvert_p\le \underset{y\in \overline{Q}}\sup~\vvvert X(s,y)\vvvert_{2p}^2.
\]
As a result, applying the inequality~\eqref{eq:lem-aux1} from Lemma~\ref{lem:aux}, one obtains
\begin{align*}
\vvvert \int_0^t \|G_d(t-\kappa(s),x,\cdot)X(s,\cdot)\|_{(\alpha)}^2 \diff s \vvvert_p&\le \int_0^t \|G_d(t-\kappa(s),x,\cdot)\|_{(\alpha)}^{2} \underset{y\in \overline{Q}}\sup~\vvvert X(s,y)\vvvert_{2p}^2 \diff s\\
&\le C_{\alpha}(T)\int_0^t (t-\kappa(s))^{-\frac{\alpha}{2}} \underset{y\in \overline{Q}}\sup~\vvvert X(s,y)\vvvert_{2p}^2 \diff s.
\end{align*}
This shows inequality~\eqref{eq:lem-tech1}.

To prove the inequality~\eqref{eq:lem-tech2}, it suffices to apply the Jensen inequality: one has
\begin{align*}
&\left(\int_0^t (t-\kappa(s))^{-\frac{\alpha}{2}}\underset{x\in \overline{Q}}\sup~\vvvert X(s,x)\vvvert_{2p}^2\diff s\right)^p\\
& \qquad \leq
\left(\int_0^t (t-\kappa(s))^{-\frac{\alpha}{2}} \diff s  \right)^{p-1} \int_0^t (t-\kappa(s))^{-\frac{\alpha}{2}}\underset{x\in \overline{Q}}\sup~\vvvert X(s,x)\vvvert_{2p}^{2p}\diff s \\
& \qquad \leq
\left(\int_0^t (t-\kappa(s))^{-\frac{\alpha}{2}} \diff s  \right)^{p-1} \int_0^t (t-\kappa(s))^{-\frac{\alpha}{2}}\underset{x\in \overline{Q}}\sup~\E[|X(s,x)|^{2p}]\diff s.
\end{align*}
Recalling that $\alpha<2$ and that $\kappa(s)\le s$ for all $s\ge 0$, one has the upper bound
\[
\displaystyle\int_0^t\left( t- \kappa(s) \right)^{-\alpha/2}\diff s \leq \int_0^t(t-s)^{-\alpha/2}\diff s\leq C_{\alpha}(T).
\]
It is then straightforward to complete the proof of the inequality~\eqref{eq:lem-tech2}.

To prove the inequality~\eqref{eq:lem-tech3}, it suffices to consider the supremum over $s\in[0,T]$: one obtains
\begin{align*}
\int_0^t (t-\kappa(s))^{-\frac{\alpha}{2}}\underset{x\in Q}\sup~\E[|X(s,x)|^{2p}]\diff s
&\leq\int_0^t (t-\kappa(s))^{-\frac{\alpha}{2}} \diff s \underset{0\le s\le t}\sup~\underset{x\in Q}\sup~\vvvert X(s,x)\vvvert_{2p}^{2p}\\
&\leq C_{p,\alpha}(T)\underset{0\le t\le T}\sup~\underset{x\in Q}\sup~\vvvert X(t,x)\vvvert_{2p}^{2p}.
\end{align*}
It is then straightforward to complete the proof of the inequality~\eqref{eq:lem-tech3}.

This concludes the proof of Lemma~\ref{lem:tech}.
\end{proof}

\section{Implementation details on the covariance matrix of the noise in dimension $d=1$}\label{app:1Dnoise}
In this appendix, we provide some details on the implementation of the discretization of the noise in Section~\ref{sec:spaceDisc} in $d=1$ dimension.
We recall that $\alpha \in (0,1)$ and that the covariance matrix $C \in \mathbb{R}^{(n-1) \times (n-1)}$ has elements given by
\begin{equation*}
C_{ij} = \int_{x_{i}}^{x_{i}+\Delta x} \int_{x_{j}}^{x_{j}+\Delta x} |x-y|^{-\alpha} \diff x \diff y,\ \text{for} \ i,j=1,\ldots, n-1,
\end{equation*}
where we set $\Delta x=\frac1n$ for the mesh size. These integrals can explicitly be computed as seen in the next lemma.
\begin{lemma}\label{lem:app_1Dnoise}
For $i < j$, one has
\begin{equation*}
C_{ij} = C_{ji}=\frac{\Delta x^{2-\alpha}}{(1-\alpha)(2-\alpha)} \left( (j - i + 1)^{2-\alpha} - 2 (j - i)^{2-\alpha} + (j - i - 1)^{2-\alpha} \right).
\end{equation*}
Furthermore, for all $i\in\{1,\ldots,n-1\}$, one has
\begin{equation*}
C_{ii} = \frac{2 \Delta x^{2-\alpha}}{(1-\alpha)(2-\alpha)},\ \text{for} \ i=1,\ldots,n-1.
\end{equation*}
\end{lemma}
\begin{proof}
We first consider the case $i < j$. This means that $x_{i} + \Delta x \leq x_{j}$ and that $y \leq x$ for all $y \in [x_{i},x_{i}+\Delta x]$ and for all $x \in [x_{j},x_{j} + \Delta x]$. One then has
\begin{equation*}
\int_{x_{j}}^{x_{j}+\Delta x} |x-y|^{-\alpha} \diff x = \int_{x_{j}}^{x_{j}+\Delta x} (x-y)^{-\alpha} \diff x = \frac{(x_{j} + \Delta x - y)^{1-\alpha}}{1-\alpha} - \frac{(x_{j} - y)^{1-\alpha}}{1-\alpha}.
\end{equation*}
Integrating the above with respect to the variable $y$, one thus obtains
\begin{equation*}
C_{ij} = \frac{1}{1-\alpha} \int_{x_{i}}^{x_{i}+\Delta x} (x_{j} + \Delta x - y)^{1-\alpha} \diff y - \frac{1}{1-\alpha} \int_{x_{i}}^{x_{i} + \Delta x} (x_{j}-y)^{1-\alpha} \diff y.
\end{equation*}
The first integral gives
\begin{equation*}
\int_{x_{i}}^{x_{i}+\Delta x} (x_{j} + \Delta x - y)^{1-\alpha} \diff y = \frac{1}{2-\alpha} \left( (x_{j}+\Delta x - x_{i})^{2-\alpha} - (x_{j}-x_{i})^{2-\alpha} \right).
\end{equation*}
The second integral gives
\begin{equation*}
\int_{x_{i}}^{x_{i} + \Delta x} (x_{j}-y)^{1-\alpha} \diff y = \frac{1}{2-\alpha} \left( (x_{j}-x_{i})^{2-\alpha} - (x_{j}-x_{i}-\Delta x)^{2-\alpha} \right).
\end{equation*}
Thus, one obtains
\begin{equation*}
C_{ij} = \frac{1}{1-\alpha} \frac{1}{2-\alpha} \left( (x_{j}+\Delta x - x_{i})^{2-\alpha} - 2 (x_{j}-x_{i})^{2-\alpha} + (x_{j}-x_{i}-\Delta x)^{2-\alpha} \right).
\end{equation*}
Finally, recalling that $x_{j} = j \Delta x$ and $x_{i} = i \Delta x$, one obtains the expression
\begin{equation*}
C_{ij} = \frac{\Delta x^{2-\alpha}}{(1-\alpha)(2-\alpha)} \left( (j - i + 1)^{2-\alpha} - 2 (j - i)^{2-\alpha} + (j - i - 1)^{2-\alpha} \right).
\end{equation*}

To compute $C_{ii}$, we first compute the integral
\begin{equation*}
\int_{x_{i}}^{x_{i} + \Delta x} |x-y|^{-\alpha} \diff x = \int_{x_{i}}^{y} (y-x)^{-\alpha} \diff x + \int_{y}^{x_{i}+\Delta x} (x-y)^{-\alpha} \diff x = \frac{1}{1-\alpha} \left( (y-x_{i})^{1-\alpha} + (x_{i}+\Delta x -y)^{1-\alpha} \right).
\end{equation*}
Thus, integrating with respect to the variable $y$ gives us that
\begin{equation*}
C_{ii} = \frac{1}{1-\alpha} \int_{x_{i}}^{x_{i}+\Delta x} (y - x_{i})^{1-\alpha} + (x_{i}+\Delta x - y)^{1-\alpha} \diff y = \frac{2 \Delta x^{2-\alpha}}{(1-\alpha)(2-\alpha)}.
\end{equation*}
This concludes the proof.
\end{proof}

\section{Implementation details on the covariance matrix of the noise in dimension $d=2$}\label{app:2Dnoise}
In this appendix, we provide some details on the implementation of the discretization of the noise in Section~\ref{sec:spaceDisc} in $d=2$ dimensions.
We recall that $\alpha \in (0,2)$ and that the covariance matrix $C \in \mathbb{R}^{(n-1)^{2} \times (n-1)^{2}}$ has elements given by
\begin{equation*}
C_{ij} = \COV (\Falpha^{n}(t,x_{i}),\Falpha^{n}(t,x_{j})) = \int_{\msquare_{{x}_{i}}} \int_{\msquare_{{x}_{j}}} ||z_{1}-z_{2}||^{-\alpha} \diff z_{1} \diff z_{2},\ \text{for} \ i,j = 1,\ldots,(n-1)^{2}.
\end{equation*}
As we are considering $d = 2$ in this section, $z_{1} \in \msquare_{{x}_{i}} \subset \mathbb{R}^{2}$ and $z_{2} \in \msquare_{{x}_{j}} \subset \mathbb{R}^{2}$. For the $i$th space grid point $x_{i} \in [0,1]^{2}$, let us write $x_i=(x_{i}^{1},x_{i}^{2})\in[0,1]^2$. $x_{j}^{1}$ and $x_{j}^{2}$ are defined in the analogous way for $x_{j} \in [0,1]^{2}$. The following lemma reduces the computation of the $4$-dimensional integrals in $C_{ij}$ to the computation of $2$-dimensional integrals.
\begin{lemma}\label{lem:app_2Dnoise}
For $i,j = 1,\ldots,(n-1)^{2}$, one has that
\begin{equation*}
C_{ij} = \begin{cases}  \displaystyle\Delta x^{4-\alpha} \int_{-1}^{1} \int_{-1}^{1} \left( (A + q_{1})^{2} + (B + q_{2})^{2} \right)^{-\alpha/2} (1-|q_{1}|) (1-|q_{2}|) \diff q_{1} \diff q_{2},\ i \neq j, \\ \Delta x^{4-\alpha} I(\alpha),\ i=j. \end{cases}
\end{equation*}
Here, $A = \frac{I_{x}(i)-I_{x}(j)}{\Delta x}$, $B = \frac{I_{y}(i)-I_{y}(j)}{\Delta x}$, and $I(\alpha) = \Delta_{2}(-\alpha)$ is defined by
\begin{equation*}
\Delta_{2}(s) = 8 \frac{(3+s) 2^{s/2}+1}{(s+2) (s+3) (s+4)} + 4 B_{2}(s) - \frac{4 (s+4)}{s+2} B_{2}(s+2),\ s \in \mathbb{C},
\end{equation*}
where
\begin{equation*}
B_{2}(s) = \frac{2}{2+s} {}_{2} F_{1}(1/2,-s/2;3/2;-1),\ s \in \mathbb{C},
\end{equation*}
and ${}_{2} F_{1}$ is the hypergeometric function, see \cite{MR2630017}.
\end{lemma}
\begin{proof}
We write $z_{1} = (x_{1},y_{1})\in[0,1]^2$ and $z_{2} = (x_{2},y_{2})\in[0,1]^2$, so that $C_{ij}$ can be written as
\begin{equation*}
C_{ij} = \int_{I_{x}(i)}^{I_{x}(i)+\Delta x} \int_{I_{y}(i)}^{I_{y}(i) + \Delta x} \int_{I_{x}(j)}^{I_{x}(j)+\Delta x} \int_{I_{y}(j)}^{I_{y}(j)+\Delta x} ((x_{1}-x_{2})^{2} + (y_{1}-y_{2})^{2})^{-\alpha/2} \diff x_{1} \diff x_{2} \diff y_{1} \diff y_{2}.
\end{equation*}
Here, we integrate over the following: $x_{1} \in [I_{x}(i),I_{x}(i)+\Delta x]$, $y_{1} \in [I_{y}(i),I_{y}(i)+\Delta x]$, $x_{2} \in [I_{x}(j),I_{x}(j)+\Delta x]$ and $y_{2} \in [I_{y}(j),I_{y}(j)+\Delta x]$. In other words, $C_{ij}$ is of the more general form
\begin{equation*}
(\ast) = \int_{a}^{a + \Delta x} \int_{b}^{b + \Delta x} \int_{c}^{c + \Delta x} \int_{d}^{d + \Delta x} \left( (x_{1}-x_{2})^{2} + (y_{1}-y_{2})^{2} \right)^{-\alpha/2} \diff x_{1} \diff y_{1} \diff x_{2} \diff y_{2},
\end{equation*}
for different values of $a,b,c,d \in [0,1]$. Here we integrate over the following: $x_{1} \in [a,a+\Delta x]$, $y_{1} \in [b,b+\Delta x]$, $x_{2} \in [c,c+\Delta x]$ and $y_{2} \in [d,d+\Delta x]$. 

Let us now introduce new variables $\tilde{x}_{1}, \tilde{x}_{2}, \tilde{y}_{1}$ and $\tilde{y}_{2}$ as follows:
\begin{equation*}
x_{1} = a + \x_{1} \Delta x,\  \diff x_{1} = \Delta x \diff \x_{1}
\end{equation*}
\begin{equation*}
x_{2} = a + \x_{2} \Delta x,\ \diff x_{2} = \Delta x \diff \x_{2}
\end{equation*}
\begin{equation*}
y_{1} = a + \y_{1} \Delta x,\  \diff y_{1} = \Delta x \diff \y_{1}
\end{equation*}
\begin{equation*}
y_{2} = a + \y_{2} \Delta x,\ \diff y_{2} = \Delta x \diff \y_{2}.
\end{equation*}
Then, for the above $4$ dimensional integral we have that
\begin{equation*}
(\ast) = \Delta x^{4} \int_{0}^{1} \int_{0}^{1} \int_{0}^{1} \int_{0}^{1} \left( (a - c + \Delta x (\x_{1}-\x_{2}))^{2} + (b - d + \Delta x (\y_{1}-\y_{2}))^{2} \right)^{-\alpha/2} \diff \x_{1} \diff \y_{1} \diff \x_{2} \diff \y_{2}.
\end{equation*}
Relabeling back to $x_{1},x_{2},y_{1},y_{2}$ and factorizing out $\Delta x^{\alpha}$, one then obtains
\begin{equation}\label{eq:app_2DRiesz}
(\ast) = \Delta x^{4-\alpha} \int_{0}^{1} \int_{0}^{1} \int_{0}^{1} \int_{0}^{1} \left( (A + (x_{1}-x_{2}))^{2} + (B + (y_{1}-y_{2}))^{2} \right)^{-\alpha/2} \diff x_{1} \diff y_{1} \diff x_{2} \diff y_{2},
\end{equation}
where $A = \frac{a-c}{\Delta x}$ and $B = \frac{b-d}{\Delta x}$.

If $i=j$ (meaning that $A = 0$ and $B = 0$), then $(\ast)$ reduces to
\begin{equation*}
(\ast) = \Delta x^{4 - \alpha} I(\alpha),
\end{equation*}
where we recall that $I(\alpha)$ is defined in the statement of the lemma. We refer to \cite{MR2630017} for more details.

If $i \neq j$ (meaning that $A \neq 0$ or $B \neq 0$), then we first consider
\begin{equation*}
(\square) = \int_{0}^{1} \int_{0}^{1} \left( (A + (x_{1}-x_{2}))^{2} + (B + (y_{1}-y_{2}))^{2} \right)^{-\alpha/2} \diff x_{1} \diff x_{2}.
\end{equation*} 
Instead of integration over $[0,1] \times [0,1]$ with respect to the measure $\diff x_{1} \diff x_{2}$, we integrate along $r \in [-\sqrt{2}/2,\sqrt{2}/2]$ (to match the domain $[0,1] \times [0,1]$) with respect to the measure $2 (\sqrt{2}/2-|r|) \diff r$
\begin{equation*}
(\square) = \int_{-\sqrt{2}/2}^{\sqrt{2}/2} \left( (A + 2 r / \sqrt{2})^{2} + (B + (y_{1}-y_{2}))^{2} \right)^{-\alpha/2} \times 2 (\sqrt{2}/2-|r|) \diff r.
\end{equation*}
We now change variables according to $q = \frac{2}{\sqrt{2}} r$ and obtain that
\begin{equation*}
(\square) = \int_{-1}^{1} \left( (A + q)^{2} + (B + (y_{1}-y_{2}))^{2} \right)^{-\alpha/2} (1-|q|) \diff q. 
\end{equation*}
Repeating the above also for the double integrals involving $y_{1},y_{2}$ in~\eqref{eq:app_2DRiesz}, gives us
\begin{equation*}
(\ast) = \Delta x^{4-\alpha} \int_{-1}^{1} \int_{-1}^{1} \left( (A + q_{1})^{2} + (B + q_{2})^{2} \right)^{-\alpha/2} (1-|q_{1}|) (1-|q_{2}|) \diff q_{1} \diff q_{2}, 
\end{equation*}
where we also added back the factor of $\Delta x^{4-\alpha}$. 
\end{proof}
In contrast to the dimension $d=1$ in Appendix~\ref{app:1Dnoise}, some integrals in $C_{ij}$ cannot be computed exactly.
On the one hand, the diagonal elements $C_{ii}$ involve the integral $I(\alpha)$ which can be computed
using matlab's command \textit{hypergeom} to evaluate the hypergeometric function.
On the other hand, the other elements of the covariance matrix need numerical approximations, where
we used the matlab command \textit{integral2}.
\end{appendix}

\bibliographystyle{plain}
\bibliography{labib}

\section{Acknowledgments}
We would like to thanks David Krantz (KTH Royal Institute of Technology) for interesting discussions on the numerical approximations
of singular quadruple integrals. The work was initiated during a research stay of the authors (DC and LQS) at the Bernoulli Center (EPFL, Switzerland)
under a Bernoulli Brainstorm program. We would like to thanks the staff of the Bernoulli Center for its support.
Part of this work was carried while DC was working for Ume{\aa} University.
Part of this work was carried thanks to the SFVE-A mobility program between France and Sweden.
The work of DC and JU was partially supported by the Swedish Research Council (VR) (projects nr. $2018-04443$ and $2024-04536$).
LQS is supported by the grant PID2021-123733NB-I00 (Ministerio de Economía y Competitividad, Spain).
The computations were performed on resources provided by the Swedish National Infrastructure
for Computing (SNIC) at HPC2N, Ume{\aa} University and by the National Academic Infrastructure for Supercomputing in Sweden (NAISS) at UPPMAX, Uppsala University, at Dardel, KTH, and at Vera, Chalmers e-Commons at Chalmers University of Technology
and partially funded by the Swedish Research Council through grant agreement no. 2022-06725.

\end{document}